\documentclass[12pt,a4paper]{article}
 
\usepackage[latin1]{inputenc}
\usepackage[english]{babel}
\usepackage[T1]{fontenc}
\usepackage{amsmath} 
\usepackage{amsfonts}
\usepackage{amssymb}
\usepackage{amsthm}
\usepackage{dsfont}
\usepackage{stmaryrd}
\usepackage{tikz}
\usepackage{todonotes}
\usetikzlibrary{matrix,arrows.meta}
\usepackage[title,titletoc,toc]{appendix}

\allowdisplaybreaks

\usepackage{hyperref}
\usepackage{xcolor}
\usepackage[a4paper,top=2.1cm,bottom=2.60cm,left=2.1cm,right=2.1cm]{geometry} 

\theoremstyle{plain}
\newtheorem{theorem}{Theorem}[section]
\newtheorem{lemma}[theorem]{Lemma}
\newtheorem{proposition}[theorem]{Proposition} 
\newtheorem{corollary}[theorem]{Corollary}

\theoremstyle{definition} 
\newtheorem{remark}[theorem]{Remark}

\begin{document}
%
%
%
%
%
%
%

\title{Shape perturbation of a nonlinear mixed problem for the heat equation}


\author{Matteo Dalla Riva\thanks{Dipartimento di Ingegneria, Universit\`a degli Studi di Palermo, Viale delle Scienze, Ed. 8, 90128 Palermo, Italy. E-mail: {matteo.dallariva@unipa.it}},
        Paolo Luzzini\thanks{Dipartimento di Scienze e Innovazione Tecnologica, Universit\`a degli Studi del Piemonte Orientale ``Amedeo Avogadro'', Viale Teresa Michel 11, 15121 Alessandria, Italy. E-mail: {paolo.luzzini@uniupo.it}}, 
        Riccardo Molinarolo\thanks{Dipartimento per lo Sviluppo Sostenibile e la Transizione Ecologica, Universit\`a degli Studi del Piemonte Orientale ``Amedeo Avogadro'', Piazza Sant'Eusebio 5, 13100, Vercelli, Italy. E-mail: {riccardo.molinarolo@uniupo.it}},
        Paolo Musolino\thanks{Dipartimento di Matematica ``Tullio Levi-Civita'', Universit\`a degli Studi di Padova, Via Trieste 63, 35121 Padova, Italy. E-mail: {paolo.musolino@unipd.it}}}

\date{\ }

\maketitle

\noindent
{\bf Abstract:}
We consider the heat equation in a domain that has a hole in its interior. We impose a Neumann condition on the exterior boundary and a nonlinear Robin condition on the boundary of the hole. The shape of the hole is determined by a suitable diffeomorphism $\phi$ defined on the boundary of a reference domain. Assuming that the problem has a solution $u_0$ when $\phi$ is the identity map, we demonstrate that a solution $u_\phi$ continues to exist for $\phi$ close to the identity map and that the ``domain-to-solution'' map $\phi\mapsto u_\phi$ is of class $C^\infty$. {Moreover, we show that the family of solutions $\{u_\phi\}_{\phi}$ is, in a sense, locally unique. Our argument relies on {tools from Potential Theory and} the Implicit Function Theorem.  Some remarks on {a} linear case complete the paper.



\section{Introduction}

This paper addresses a nonlinear mixed problem for the heat equation. Our aim is to study how the solutions depend on perturbations of the domain of definition. There are many real-world scenarios where the relationship between the properties of an object and its shape is relevant. This occurs, for example, when one searches for designs that optimize specific properties of an object. In mathematics, this problem led to the development of the field known as ``shape optimization,'' for which we refer the reader to the monographs of Soko\l owski and Zol\'esio \cite{SoZo92}, Novotny and Soko\l owski \cite{NoSo13}, and Henrot and Pierre \cite{HePi18}, among others.

In many cases, shape optimization problems require understanding the regularity of the map associating the domain's shape--or other parameters--with a solution.  For example, knowing that this map is continuous implies that we can control small changes in the solution with small changes in the parameters. Differentiability enables the use of tools from differential calculus to determine optimal configurations. Smoothness and analyticity are even stronger properties. Smoothness allows the approximation of the solution using its Taylor polynomials with any desired degree of accuracy, while analyticity allows the solution to be expressed as a convergent power series of the perturbation parameters.

This paper investigates precisely this type of {problems}. Our specific aim is to demonstrate that, for the {nonlinear problem for the heat equation} under consideration, the  ``domain-to-solution'' map is smooth. To that aim, we utilize the Functional Analytic Approach introduced by Lanza de Cristoforis (see, e.g., \cite{DaLaMu21} and references therein) and accordingly, we heavily rely on potential theoretic methods. In particular, we will {exploit} the results in \cite{DaLu23} concerning the smooth dependence of the heat layer potential on the shape of the integration support.

The use of potential theory in the context of shape perturbation problems is not a novelty{. Indeed}, many authors, especially in the {context} of elliptic equations, have adopted this strategy. A preliminary step {in} this approach involves the shape sensitivity analysis of the layer potentials.  For example,  Potthast \cite{Po94,Po96a} proved that layer potentials for the Helmholtz equation are Fr\'echet differentiable functions of the support of integration. Similar results have been obtained for a variety of equations, including the Stokes system of fluid dynamics and the Lam\'e equations of elasticity. The reader may for {example refer to} the works of Charalambopoulos \cite{Ch95}, Costabel and Le Lou\"er \cite{CoLe12a},  Haddar and Kress \cite{HaKr04}, Hettlich \cite{He95}, and Kirsch \cite{Ki93}. 

Results proving regularities beyond differentiability are rare in this context, but exceptions exist, notably in the works of Lanza de Cristoforis and his collaborators. We have already mentioned the smoothness results of \cite{DaLu23} for the heat layer potentials, we should also mention the extension to the periodic case obtained in \cite{DaLuMoMu24} and the previous analyticity results for the Laplace operator in \cite{LaRo04} and for more general elliptic differential operators in \cite{DaLa10}.

Apart from {the results stemmed from Lanza de Cristoforis'  work}, another recent approach to derive analyticity results is through ``shape holomorphy.'' This method has been employed by   Henr\'iquez and Schwab in \cite{HeSc21} to study the Calder\'on projector for the Laplacian in $\mathbb{R}^2$, and by Pinto, Henr\'iquez, and Jerez-Hanckes in \cite{PiHeJe24} for boundary integral operators on multiple open arcs.

We also observe that almost all literature on shape sensitivity is about elliptic problems and much less is available for parabolic problems. Exceptions include the results in \cite{DaLu23,DaLuMoMu24}, which we have already mentioned, and the works of Chapko, Kress and Yoon \cite{ChKrYo98,ChKrYo99} and Hettlich and Rungell \cite{HeRu01}, where the authors prove  Fr\'echet shape-differentiability of the solutions and explore applications to certain inverse problems in heat conduction. 

In this context, the present paper aims to build upon the research of \cite{DaLu23,DaLuMoMu24} and address a gap in the literature. 

The specific boundary value problem under consideration is a mixed boundary value problem for the heat equation in a perforated domain, with a Neumann condition on a fixed outer boundary and a nonlinear Robin condition on a perturbed inner boundary. Mixed Neumann-Robin boundary value problems for the heat equations have been analyzed by several authors, due to their applications. For instance, Bacchelli, Di Cristo, Sincich, and Vessella \cite{BaDiSiVe14}, as well as {Nakamura and Wang \cite{NaWa15,NaWa17}}, have investigated such problems in the context of inverse problems. Nonlinear boundary conditions for the heat equation have also been extensively studied: {see, e.g., Friedman \cite[Chapter~7]{Fr08} for a discussion}. Generalizations of these problems can be found in the more recent work by Biegert and Warma \cite{BiWa09}.

To define our boundary value problem, we take $\alpha \in \mathopen]0,1[$,  a natural number
\[
n \in \mathbb{N} \setminus \{0, 1\}\, ,
\]
 and two sets $\Omega$ and $\omega$ that satisfy the following condition:
\begin{equation}\label{introsetconditions}
	\begin{split}
		&\mbox{$\Omega$, $\omega$ are bounded open connected subsets of $\mathbb{R}^n$ of class $C^{1,\alpha}$,} 
		\\
		&\mbox{with connected exteriors  $\Omega^- : = \mathbb{R}^n\setminus \overline{\Omega}$ and $\omega^- : =\mathbb{R}^n\setminus \overline{\omega}$,}
		\\
		 &\mbox{and such that $\overline{\omega}\subseteq\Omega$}.
	\end{split}
\end{equation}
Here above $\overline{\cdot}$ denotes the closure of a set and, for the definition of Schauder spaces and domains of class $C^{1,\alpha}$, we refer to Gilbarg and
Trudinger \cite[pp.~52, 95]{GiTr83}. {As done in \eqref{introsetconditions}, for an open set $\tilde{\Omega}$ in $\mathbb{R}^n$, we denote by $\tilde{\Omega}^-$ its exterior $\mathbb{R}^n \setminus \overline{\tilde{\Omega}}$.}

Our boundary value problem will be defined on a perforated domain obtained by removing from $\Omega$ a perturbed copy of the set $\omega$. Specifically, we will perturb $\omega$ with a diffeomorphism $\phi$ from the class 
\begin{equation}\label{A_omega}
\mathcal{A}_{\partial\omega} := \left\{ \phi \in  C^{1,\alpha}(\partial\omega, \mathbb{R}^n): \, \phi \text{ injective}, \, d\phi(y) \text{ injective for all } y \in \partial\omega \right\}\,.
\end{equation}
According to the Jordan-Leray Separation Theorem (cf.~Deimling \cite[Thm.  5.2, p. 26]{De85}),  $\phi(\partial\omega)$ splits $\mathbb{R}^n$ into exactly two open connected components, one bounded and one unbounded. We denote by 
\[
\omega[\phi]
\] 
the bounded open
connected component of $\mathbb{R}^n \setminus \phi(\partial\omega)$, and we clearly have that
\[
\partial \omega[\phi]=\phi(\partial\omega)\,.
\]
However, to ensure that $\overline{\omega[\phi]}$ remains within $\Omega$ and thus define the perturbed perforated domain, we introduce the set
\begin{equation}\label{A^Omega_omega}
    \mathcal{A}^{\Omega}_{\partial\omega} := \left\{ \phi \in \mathcal{A}_{\partial\omega} : \overline{\omega[\phi]} \subseteq \Omega  \right\}\,,
\end{equation}
which contains all $\phi \in \mathcal{A}_{\partial\omega}$ such that the closure of $\omega[\phi]$ is a subset of $\Omega$. For $\phi\in \mathcal{A}^{\Omega}_{\partial\omega}$, we can consider the perturbed perforated domain $\Omega\setminus\overline{\omega[\phi]}$ {(see Figure \ref{fig:1}).} 
\begin{figure}[!htb]
\centering
\includegraphics[width=3.2in]{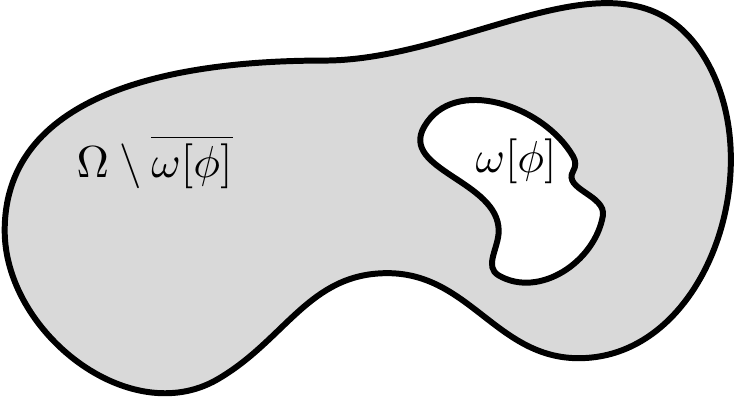}
\caption{{\it The sets $\Omega\setminus\overline{\omega[\phi]}$ and $\omega[\phi]$.}}\label{fig:1} 
\end{figure}

We also note that the identity function $\mathrm{id}_{\partial\omega}$ of $\partial\omega$ belongs to $\mathcal{A}^\Omega_{\partial\omega}$ and,
for convenience, we set 
\[
\phi_0 := \mathrm{id}_{\partial\omega}.
\]
{Accordingly,
\[\Omega\setminus \overline{\omega[\phi_0]}=\Omega \setminus \overline{\omega}\, .\]}
We observe that $\partial(\Omega\setminus \overline{\omega[\phi]})$ has two connected components:  {$\partial\Omega$ that remains fixed, and  $ \partial \omega[\phi]$ that} depends on $\phi$. To define the boundary conditions on $\partial(\Omega\setminus \overline{\omega[\phi]})$ we will use two functions $G$ and $f$ such that:
{\begin{equation}\label{introfunconditions}
\begin{split}
&G \in C^0([0,T] \times {\Omega} \times \mathbb{R})\, ,\quad \quad{G(0,x,0)=0}\quad\text{for all }x\in{\Omega}\,,\\
&\text{and the function}\  [0,T] \times \partial\omega \ni (t,x)\mapsto G(t,{\phi(x)},u(t,x))\  \text{belongs}\\ 
&\text{to $C^{\frac{\alpha}{2};\alpha}([0,T] \times \partial\omega)$ for all $u \in C^{\frac{1+\alpha}{2};1+\alpha}([0,T] \times \partial\omega)$ {and $\phi \in\mathcal{A}^\Omega_{\partial\omega}$}}\,,\\
&f \in C_{0}^{\frac{\alpha}{2}; \alpha}([0,T] \times \partial\Omega).
\end{split}
\end{equation}}
The function $G$ determines the nonlinear Robin-type condition on the inner boundary $[0,T]\times \phi(\partial \omega)$, while $f$ serves as the Neumann datum on the outer boundary $[0,T]\times \partial \Omega$ (see Section \ref{s:prel} for the definition of $C_{0}^{\frac{\alpha}{2}; \alpha}([0,T] \times \partial\Omega)$). {Conditions ensuring the validity of assumption \eqref{introfunconditions} are discussed right after equation \eqref{condition NG*}. Here we just mention that one could take, for example,   any polynomial function $G(t,x,\xi)=a_0(t)b_0(x)+a_1(t)b_1(x)\xi+\ldots+a_k(t)b_k(x)\xi^k$ with $a_0(0)=0$ and where, for each $j \in \{1,\dots, k\}$, we have $a_j \in C^{\frac{\alpha}{2}}([0,T])$ and $b_j\in C^{1}(\Omega)$.  }

Then we consider the following {nonlinear} mixed  boundary value problem for a function $u \in C_{0}^{\frac{1+\alpha}{2}; 1+\alpha}([0,T] \times (\overline{\Omega} \setminus \omega[\phi]))$:

{\begin{equation}\label{princeqpertu}
\begin{cases}
    \partial_t u - \Delta u = 0 & \quad\text{in } ]0,T] \times (\Omega \setminus \overline{\omega[\phi]}), 
    \\
    \frac{\partial}{\partial \nu_{\Omega}}  u(t,x) = f (t,x)& \quad \forall (t,x)\in [0,T] \times \partial \Omega, 
    \\
    \frac{\partial}{\partial {\nu_{\omega[\phi]}}}  u (t,x) = G(t,{x},u(t,x)) & \quad \forall (t,x)\in  [0,T] \times \partial \omega[\phi],
    \\
    u(0,\cdot)=0 & \quad \text{in } \overline{\Omega} \setminus \omega[\phi],
    \end{cases}
\end{equation}}
where {$\nu_\Omega$ and ${\nu_{\omega[\phi]}}$ respectively denote the outward unit normal vector field to $\partial\Omega$
and to $\partial \omega[\phi]=\phi(\partial \omega)$}. 

 Our goal is now as follows: Under suitable assumptions, which include the existence of a solution $u_0$ of \eqref{princeqpertu} for $\phi=\phi_0$, we aim to establish the existence of solutions $u_\phi$ for $\phi$ sufficiently close to the identity. This scenario corresponds to the case where the perturbed hole $\omega[\phi]$ is close to $\omega$. Subsequently, we will demonstrate that the map associating $\phi$ with $u_\phi$ is smooth {and that the family of solutions $\{u_\phi\}_{\phi}$ is locally unique in a sense that will be clarified later on. Finally, in the linear case, we will show the smoothness of the ``Neumann-to-Dirichlet'' operator upon  the shape parameter $\phi$ and the boundary condition.}

The paper is organized as follows: Section \ref{s:prel} provides preliminaries on parabolic Schauder spaces and potential theory for the heat equation. In Section \ref{sec:pert}, we transform problem \eqref{princeqpertu} into an equivalent nonlinear system of integral equations, which we analyze using the Implicit Function Theorem. Additionally, we present our main results on the smoothness of the {``domain-to-solution''} map (Theorem \ref{thm:smoothrep}) {and on the local uniqueness of the family of solutions $\{u_\phi\}_\phi$ (Proposition \ref{prop:locrep}). Finally, some remarks on {a linear problem} in Section \ref{sec:lin} conclude the paper.} 

\section{Preliminaries}\label{s:prel}

In this section, we provide some preliminaries, particularly focusing on the potential theory for the heat equation. In the subsequent sections, we will utilize the single layer potential to convert problem \eqref{princeqpertu} into a system of integral equations, which we will then analyze using the Implicit Function Theorem.

We begin by recalling the definition of parabolic Schauder spaces. {Let $\alpha \in \mathopen]0,1[$, $T>0$ and {$\Omega$ be an open subset of $\mathbb{R}^n$}. Then $C^{\frac{\alpha}{2};\alpha}([0, T] \times \overline{\Omega})$ denotes the space of
bounded continuous functions $u$ from $[0, T] \times \overline\Omega$ to $\mathbb{R}$ such that
\begin{align*}
    \|u\|_{C^{\frac{\alpha}{2};\alpha}([0, T] \times \overline{\Omega})} :=& \sup_{[0, T] \times \overline\Omega} |u| + \sup_{\substack{t_1,t_2 \in [0,T] \\ t_1 \neq t_2}} \,\sup_{x \in \overline\Omega} \frac{|u(t_1,x) - u(t_2,x)|}{|t_1-t_2|^\frac{\alpha}{2}}
    \\
    & + \sup_{t \in [0,T]} \,\sup_{\substack{x_1,x_2 \in \overline\Omega \\ x_1 \neq x_2}} \frac{|u(t,x_1) - u(t,x_2)|}{|x_1-x_2|^\alpha} < +\infty.
\end{align*}}
{Also,} $C^{\frac{1+\alpha}{2}; 1+\alpha}([0, T] \times \overline{\Omega})$ denotes the space of
bounded continuous functions $u$ from $[0, T] \times \overline{\Omega}$ to $\mathbb{R}$ which are continuously differentiable
with respect to the space variables and such that
\begin{align*}
    \|u\|_{C^{\frac{1+\alpha}{2}; 1+\alpha}([0, T] \times \overline{\Omega})} :=& \sup_{[0, T] \times \overline{\Omega}} |u| + \sum_{i=1}^{n} \|\partial_{x_i} u\|_{C^{\frac{\alpha}{2}; \alpha}([0, T] \times \overline{\Omega})} 
    \\
    & + \sup_{\substack{t_1,t_2 \in [0,T] \\ t_1 \neq t_2}} \,\sup_{x \in \overline{\Omega}} \frac{|u(t_1,x) - u(t_2,x)|}{|t_1-t_2|^{\frac{1+\alpha}{2}}} < +\infty.
\end{align*}
If $\Omega$ is an open subset of $\mathbb{R}^n$ of class $C^{1,\alpha}$, we can use the local {parametrization} of $\partial \Omega$ to define the space
$C^{\frac{1+\alpha}{2}; 1+\alpha}([0, T] \times \partial \Omega)$ in the natural way. In a similar way, we can define the spaces $C^{j,\alpha}(\mathcal{M})$ and $C^{\frac{j+\alpha}{2}; j+\alpha}([0, T] \times \mathcal{M})$, $j \in \{0,1\}$ on a manifold $\mathcal{M}$ of class $C^{j,\alpha}$ imbedded in $\mathbb{R}^n$ (see \cite[Appendix A]{DaLu23}). In essence, a function of class $C^{\frac{j+\alpha}{2}; j+\alpha}$ is $\left(\frac{j+\alpha}{2}\right)$-H\"older continuous in the time variable,
and $(j,\alpha)$-Schauder regular in the space variable. 

We use the subscript $0$ to denote a subspace consisting of functions that are
zero at $t = 0$. Namely,
\[
{C_0^{\frac{\alpha}{2};\alpha}([0, T] \times \overline{\Omega})}:= 
\Big\{ u \in {C^{\frac{\alpha}{2};\alpha}([0, T] \times \overline{\Omega})} \,:\, u(0,x) = 0 \quad \forall x \in \Omega\Big\}.
\]
Then $C_0^{\frac{j+\alpha}{2}; j+\alpha}([0, T] \times \overline{\Omega})$, $C_0^{\frac{j+\alpha}{2}; j+\alpha}([0, T] \times \partial \Omega)$ and  $C_0^{\frac{j+\alpha}{2}; j+\alpha}([0, T] \times \mathcal{M})$ {with $j \in \{0,1\}$}
are similarly defined.

For functions in parabolic Schauder spaces, the partial derivative $D_x$ with respect to the space variable $x$ will be denoted by the gradient $\nabla$, while we will maintain the notation $\partial_t$ for the derivative with respect to the time variable $t$.

For a comprehensive introduction to parabolic Schauder spaces we refer the reader to
classical monographs on the field, for example Lady\v{z}enskaja, Solonnikov, and Ural'ceva
\cite[Chapter 1]{LaSoUr68} (see also \cite{LaLu17,LaLu19} {for more references}).

In our boundary value problem, the nonlinear Robin condition is defined through a superposition operator associated with the function $G$. Therefore, if $\omega$ is as in \eqref{introsetconditions}, we now introduce a notation for such operators: {If $G$ is a function from $[0,T] \times \Omega \times \mathbb{R}$ to $\mathbb{R}$, then we denote by $\mathcal{N}_G$ the nonlinear superposition operator that maps a pair $(\phi,u)$, where $\phi \in \mathcal{A}_{\partial \omega}^\Omega$ and $u$ is a function from $[0,T]\times \partial\omega$, to the function $\mathcal{N}_G(\phi, u)$ defined by
\[
\mathcal{N}_G(\phi,u) (t,x) := G(t,\phi(x),u(t,x)) \quad \text{for all } (t,x) \in [0,T]\times \partial\omega.
\]
Here, the letter ``$\mathcal{N}$'' stands for ``Nemytskii operator.''}

Now we turn to present some well-known facts on the {single layer heat potential}. For proofs
and detailed references we refer to Lady\v{z}enskaja, Solonnikov, and Ural'ceva
\cite[Chapter 4]{LaSoUr68}. 

First of all, since layer potentials are integral operators whose kernel is a fundamental solution or its derivatives, we define the function $S_{n} : \mathbb{R}^{1+n} \setminus
\{(0,0)\}\to \mathbb{R}$  by
\[
S_{n}(t,x):=
\left\{
\begin{array}{ll}
\frac{1}{(4\pi t)^{\frac{n}{2}} }e^{-\frac{|x|^{2}}{4t}}&{\mathrm{if}}\ (t,x)\in \mathopen]0,+\infty[ \times{\mathbb{R}}^{n}\,, 
\\
0 &{\mathrm{if}}\ (t,x)\in (\mathopen]-\infty,0]\times{\mathbb{R}}^{n})\setminus\{(0,0)\}.
\end{array}
\right.
\]
It is well known that $S_n$ is a fundamental solution of the heat operator $\partial_t-\Delta$ in $\mathbb{R}^{1+n} \setminus \{(0,0)\}$.

We are now in the position to introduce the single layer heat potential. Let $\alpha \in \mathopen]0,1[$ and $T>0$. Now let  $\Omega$  be an open bounded subset of $\mathbb{R}^n$ of class $C^{1,\alpha}$. For a density $\mu \in L^\infty\big([0,T] \times \partial\Omega\big)$, the single layer heat potential is defined as
\begin{equation*} 
    v_{\Omega} [\mu](t,x) := \int_{0}^{t} \int_{\partial \Omega} S_{n}(t-\tau,x-y) \mu(\tau, y)\,d\sigma_y d\tau \quad \forall\,(t,x) \in [0, T] \times \mathbb{R}^n.
\end{equation*}
Moreover, we set
\begin{equation*}
    V_{\partial\Omega}[\mu] := v_{\Omega}[\mu]_{|[0,T]\times \partial\Omega} 
\end{equation*}
and
\begin{equation*}
\begin{split}
    W^*_{\partial \Omega}[\mu](t,x) := \int_{0}^t\int_{\partial\Omega} 
    \frac{\partial}{\partial \nu_\Omega(x)} S_{n}(t-\tau,x-y) \mu(\tau,y)\,&d\sigma_yd\tau \\
    & \forall\,(t,x) \in [0,T] \times \partial\Omega\, .
\end{split}
\end{equation*}

The map $V_{\partial\Omega}[\mu]$ is the trace of the single layer heat potential on $[0,T]\times \partial\Omega$, whereas the map $W^*_{\partial\Omega}[\mu]$ is associated with the normal
derivative of the single layer heat potential on $[0,T]\times \partial\Omega$ (see Theorem \ref{thmsl} (iii) below).

Layer heat potentials enjoy properties similar to those of their standard elliptic counterpart. We collect those related to the single layer potential in the following theorem.

\begin{theorem}\label{thmsl}
Let $\alpha \in \mathopen]0,1[$ and $T>0$. Let $\Omega$ be a bounded open subset of $\mathbb{R}^n$ of class $C^{1,\alpha}$. Then the following statements hold.
\begin{itemize}

\item[(i)] Let $\mu \in L^\infty([0,T] \times \partial\Omega)$. Then the function $v_{\Omega}[\mu]$ is continuous and 
$v_{\Omega}[\mu] \in C^\infty(]0,T[ \times (\mathbb{R}^n \setminus \partial\Omega))$. 
Moreover $v_{\Omega}[\mu]$ solves the heat equation 
in $]0,T]\times (\mathbb{R}^n \setminus \partial\Omega)$.

\item[(ii)] Let $v_\Omega^+[\mu]$ and $v_\Omega^-[\mu]$ denote 
the restrictions of $v_\Omega[\mu]$ to $[0,T] \times \overline{\Omega}$ and to $[0,T]\times \overline{\Omega^-}$, respectively. Then, the map from  $C_0^{\frac{\alpha}{2};  \alpha}([0,T] \times \partial\Omega)$ to  $C_{0}^{\frac{1+\alpha}{2}; 1+\alpha}([0,T] \times \overline{\Omega})$ that takes $\mu$ to $v_{\Omega}^+[\mu]$ is linear and continuous. If $R>0$ is such that $\overline{\Omega}$ is contained in the ball $B(0,R)$ of center $0$ and radius $R$, then the map from  $C_0^{\frac{\alpha}{2};  \alpha}([0,T] \times \partial\Omega)$ to  $C_{0}^{\frac{1+\alpha}{2}; 1+\alpha}([0,T] \times (\overline{B(0,R)}\setminus \Omega^-))$ that takes $\mu$ to $v_{\Omega}^-[\mu]_{|[0,T] \times (\overline{B(0,R)}\setminus \Omega^-)}$ is also linear and continuous.

\item[(iii)] Let $\mu \in C_0^{\frac{\alpha}{2};  \alpha}([0,T] \times \partial\Omega)$. Then the following jump relations hold: 
\[
\frac{\partial}{\partial \nu_\Omega}v_{\Omega}^\pm[\mu](t,x)  = \pm \frac{1}{2}\mu(t,x)
 +W_{\partial\Omega}^*[\mu](t,x),
 \]
 for all $(t,x) \in [0,T] \times \partial\Omega$.
\item[(iv)] The operator $V_{\partial \Omega}$ is an isomorphism from the space $C_0^{\frac{\alpha}{2}; \alpha}([0,T] \times \partial\Omega)$ to $C_0^{\frac{1+\alpha}{2}; 1+\alpha}([0,T] \times \partial\Omega)$.
\end{itemize}
\end{theorem}

\begin{proof}
For the proof of points (i)--(iii) we refer, for instance, to \cite{DaLu23,LaLu17,LaLu19}. {For the proof of (iv), one can simplify the argument of the proof of Brown \cite[Prop.~6.2]{Br89} (which deals with the more involved case of Lipschitz domains) in combination with the regularity results for layer potential operators in Schauder spaces of \cite{LaLu17,LaLu19}. For invertibility results of the trace of the single layer potential for the heat equation, see also Costabel \cite[Corollary~3.13]{Co90} and Brown \cite[Theorem~4.18]{Br90}.} {We emphasize here that in order to prove (iv) one uses the fact that if $\mu \in C_0^{\frac{\alpha}{2}; \alpha}([0,T] \times \partial\Omega)$ is such that
\[
V_{\partial \Omega}[\mu]=0 \qquad \text{on $[0,T] \times \partial\Omega$},
\]
then by the uniqueness of the (interior) initial Dirichlet problem for the heat equation one has
\[
v_{\Omega}^+[\mu]=0 \qquad \text{in $[0,T] \times \overline{\Omega}$}, 
\]
and accordingly
\[
\frac{\partial}{\partial \nu_\Omega}v_{\Omega}^+[\mu]=0 \qquad \text{on $[0,T] \times \partial\Omega$}.
\]
Similarly, condition 
\[
V_{\partial \Omega}[\mu]=0 \qquad \text{on $[0,T] \times \partial\Omega$},
\] 
and the exponential decay at infinity of $v_{\Omega}^-[\mu]$ allow to show that
\[
v_{\Omega}^-[\mu]=0 \qquad \text{in $[0,T] \times \overline{\Omega^-}$}, 
\]
and accordingly
\[
\frac{\partial}{\partial \nu_\Omega}v_{\Omega}^-[\mu]=0 \qquad \text{on $[0,T] \times \partial\Omega$}
\]
(cf.~the proof of Brown \cite[Prop.~6.2]{Br89} and Costabel \cite[p.~522]{Co90}). Then the jump formulas for $\frac{\partial}{\partial \nu_\Omega}v_{\Omega}^\pm[\mu]$ imply that $\mu=0$ (see also \cite[Thm.~2.7]{Mo25} for a detailed proof). 
}
\qed \end{proof}

Fredholm theory is an important tool for analyzing the integral equations associated with the boundary behavior of layer potentials and their normal derivatives. To apply this theory, we need some compactness results on the embedding of parabolic Schauder spaces, which we state in the following Proposition \ref{Ascoli Arzela cons prop}. The proof of this proposition is well-known and relies on the Ascoli-Arzel\`a Theorem.

\begin{proposition}\label{Ascoli Arzela cons prop}
Let $\alpha \in \mathopen]0,1[$ and $T>0$. Let $\Omega$ be a bounded open subset of $\mathbb{R}^n$ of class $C^{1,\alpha}$. Then, the embeddings of $C_0^{\frac{\alpha}{2}; \alpha}([0,T] \times \partial\Omega)$ into $C_0^{\frac{\gamma}{2}; \gamma}([0,T] \times \partial\Omega)$ for any $\gamma \in [0,\alpha[$ and of $C_0^{\frac{1+\alpha}{2}; 1+\alpha}([0,T] \times \partial\Omega)$ into $C_0^{\frac{\alpha}{2}; \alpha}([0,T] \times \partial\Omega)$ are compact.
\end{proposition}

From Proposition \ref{Ascoli Arzela cons prop}, combined with Theorem \ref{thmsl} (iv) and \cite[Theorem 4.6]{LaLu19}, we readily derive the following Theorem \ref{thm V and W*} on the mapping and compactness properties of the operators $V_{\partial\Omega}$ and $W^\ast_{\partial\Omega}$.

\begin{theorem}\label{thm V and W*}
Let $\alpha \in \mathopen]0,1[$ and $T > 0$. Let $\Omega$ be a bounded open subset of $\mathbb{R}^n$ of class $C^{1,\alpha}$. Then, the operators $V_{\partial\Omega}$ and $W^*_{\partial\Omega}$ are compact from $C_0^{\frac{\alpha}{2};  \alpha}([0,T] \times \partial\Omega)$ to itself.
\end{theorem}

Problem \eqref{princeqpertu} is defined in a perforated domain, which is a connected set obtained by removing a part of a larger domain. We will need some results for functions defined in these specific types of sets. Therefore, in the remaining part of this section, we focus on perforated domains. We consider the following assumption:

\begin{equation}\label{assOmtildom}
	\begin{split}
		&\mbox{$\Omega$ and $\tilde{\omega}$ are bounded open connected subsets of $\mathbb{R}^n$ of class $C^{1,\alpha}$,} 
		\\
		&\mbox{with connected exteriors  {$\Omega^-:=\mathbb{R}^n\setminus \overline{\Omega}$ and $\tilde{\omega}^-:=\mathbb{R}^n\setminus \overline{\tilde{\omega}}$}} 
		\\
		&\mbox{and such that $\overline{\tilde{\omega}}\subseteq\Omega$}.
	\end{split}
\end{equation}

{By Theorem \ref{thmsl} (iv) and by the uniqueness of the solution of the Dirichlet problem for the heat equation, we deduce the following Lemma \ref{lemma rappr}, where we see that solutions of the heat equation that are in $C_{0}^{\frac{1+\alpha}{2}; 1+\alpha}([0,T] \times (\overline{\Omega} \setminus \tilde{\omega}))$ can be represented as a sum of single layer potentials.}

\begin{lemma}\label{lemma rappr} 
Let $\alpha \in \mathopen]0,1[$ and $T>0$. Let $\Omega$ and $\tilde{\omega}$ be as in assumption \eqref{assOmtildom}. Then the map from $C_0^{\frac{\alpha}{2};  \alpha}([0,T] \times \partial\Omega) \times C_0^{\frac{\alpha}{2};  \alpha}([0,T] \times \partial \tilde{\omega})$ to the space
\begin{equation*}
    \left\{u\in C_{0}^{\frac{1+\alpha}{2}; 1+\alpha}([0,T] \times (\overline{\Omega} \setminus \tilde{\omega})) \,:\, \partial_t u  - \Delta u =0 \quad \text{in } {\mathopen]0,T]} \times {(\Omega \setminus \overline{\tilde{\omega}})} \right\}
\end{equation*}
that takes a pair $(\mu,\eta)$ to the function 
\begin{equation}\label{U}
	{u_{\Omega,\tilde\omega}[\mu,\eta]:=}(v^+_{\Omega} [\mu] + v^-_{\tilde{\omega}}[\eta])_{| [0,T] \times (\overline{\Omega} \setminus \tilde{\omega})}
\end{equation}
is a bijection.
\end{lemma}

To apply the Implicit Function Theorem, we need to examine a linearization of problem \eqref{princeqpertu}, i.e., a linear mixed Neumann-Robin boundary value problem. For this, we need the following uniqueness result, whose proof follows, for example, from Liebermann \cite[Corollary 5.4]{Li96} (see also Friedman \cite[Lemma 2, p. 146]{Fr08} and Ladyzhenskaja, Solonnikov, and Ural'ceva \cite[Chapter 4, \S 5]{LaSoUr68}). 

\begin{lemma}\label{beta prob lemma}
    Let $\alpha \in \mathopen]0,1[$ and $T>0$. Let $\Omega$ and $\tilde{\omega}$ be as in assumption \eqref{assOmtildom}.  Let $\beta\in C^{\frac{\alpha}{2};  \alpha}([0,T] \times \partial\tilde{\omega})$. Then the unique solution in $C_{0}^{\frac{1+\alpha}{2}; 1+\alpha}([0,T] \times (\overline{\Omega} \setminus \tilde{\omega}))$ of problem
	{\begin{equation}\label{beta problem eq}
    \begin{cases}
	\partial_t u - \Delta u = 0 & \quad\text{in } \mathopen]0,T] \times (\Omega \setminus \overline{\tilde{\omega}}), 
	\\
   \frac{ \partial }{ \partial\nu_{\Omega} }u(t,x) = 0 & \quad \forall (t,x)\in  [0,T] \times \partial \Omega, 
	\\
	   {\frac{ \partial }{\partial\nu_{\tilde{\omega}} }u}(t,x)  - \beta(t,x) u(t,x) = 0 &\quad  \forall (t,x)\in [0,T] \times \partial \tilde{\omega},
    \\
    u(0,\cdot)=0 & \quad \text{in } \overline{\Omega} \setminus \tilde{\omega},
    \end{cases}
    \end{equation}}
	is $u=0$.
\end{lemma}

We conclude this section with the following proposition, where we investigate {an} auxiliary boundary operator, $\mathcal{J}_\beta$, associated with the linear Neumann-Robin boundary value problem \eqref{beta problem eq}. We are going to use $\mathcal{J}_\beta$ to analyze the integral formulation of problem \eqref{princeqpertu}.

\begin{proposition}\label{prop J_beta}
      Let $\alpha \in \mathopen]0,1[$ and $T>0$. Let $\Omega$ and $\tilde{\omega}$ be as in assumption \eqref{assOmtildom}. Let $\beta\in C^{\frac{\alpha}{2};  \alpha}([0,T] \times \partial\tilde{\omega})$. Let $\mathcal{J}_\beta = (\mathcal{J}_{\beta,1}, \mathcal{J}_{\beta,2})$ be the map from $C_0^{\frac{\alpha}{2};  \alpha}([0,T] \times \partial\Omega) \times C_0^{\frac{\alpha}{2};  \alpha}([0,T] \times \partial\tilde{\omega})$ to itself that takes a pair $(\mu,\eta)$ to the pair $\mathcal{J}_\beta[\mu,\eta]:=(\mathcal{J}_{\beta,1}[\mu,\eta],\mathcal{J}_{\beta,2}[\mu,\eta])$ defined by
	\begin{equation}\label{J_beta eq}
	\begin{aligned}
	\mathcal{J}_{\beta,1}[\mu,\eta] &:= \left( \frac{1}{2} I + W^\ast_{\partial\Omega} \right) [\mu] + \nu_{\Omega} \cdot \nabla  v^-_{\tilde{\omega}}[\eta]_{|[0,T] \times\partial\Omega} {\quad \text{on $[0,T]\times \partial \Omega$}},
	\\
	\mathcal{J}_{\beta,2}[\mu,\eta] &:= \left( -\frac{1}{2} I + W^\ast_{\partial\tilde{\omega}} \right) [\eta] +
    \nu_{\tilde{\omega}} \cdot \nabla  v^+_{\Omega}[\mu]_{|[0,T] \times\partial\tilde{\omega}}
   \\
   & \qquad  - \beta (v^+_{\Omega}[\mu]_{|[0,T] \times\partial\tilde{\omega}} + V_{\partial\tilde{\omega}}[\eta]) {\quad \text{on $[0,T]\times \partial \tilde{\omega}$}}.  
	\end{aligned}
	\end{equation}
	Then  $\mathcal{J}_{\beta}$ is a linear homeomorphism from $C_0^{\frac{\alpha}{2};  \alpha}([0,T] \times \partial\Omega) \times C_0^{\frac{\alpha}{2};  \alpha}([0,T] \times \partial\tilde{\omega})$ to itself.
\end{proposition}

\begin{proof}
Let $\bar{\mathcal{J}} := (\bar{\mathcal{J}}_1, \bar{\mathcal{J}}_2)$ be the linear operator from $C_0^{\frac{\alpha}{2};  \alpha}([0,T] \times \partial\Omega) \times C_0^{\frac{\alpha}{2};  \alpha}([0,T] \times \partial\tilde{\omega})$ to itself defined by
\begin{equation*}
    \bar{\mathcal{J}}_1[\mu,\eta] := \frac{1}{2} \mu, \quad \bar{\mathcal{J}}_2[\mu,\eta] := -\frac{1}{2} \eta,
\end{equation*}
for all $(\mu,\eta) \in C_0^{\frac{\alpha}{2};  \alpha}([0,T] \times \partial\Omega) \times C_0^{\frac{\alpha}{2};  \alpha}([0,T] \times \partial\tilde{\omega})$. Clearly $\bar{\mathcal{J}}$ is a linear homeomorphism.

Moreover, {by Theorem \ref{thm V and W*}}, by the
results of \cite[Lemmas A.2, A.3]{DaLu23} on non-autonomous composition operators and on
time-dependent integral operators with non-singular kernels, {and  by the {compactness  of the} embeddings of parabolic Schauder spaces of {Theorem \ref{Ascoli Arzela cons prop}}}, we deduce that the map from $C_0^{\frac{\alpha}{2};  \alpha}([0,T] \times \partial\Omega) \times C_0^{\frac{\alpha}{2};  \alpha}([0,T] \times \partial\tilde{\omega})$ to itself that takes a pair $(\mu,\eta)$ to the pair $\mathcal{J}^C_{\beta}[\mu,\eta] {= (\mathcal{J}^C_{\beta,1}[\mu,\eta] ,\mathcal{J}^C_{\beta,1}[\mu,\eta])}$ defined by
\[
\begin{split}
	\mathcal{J}^C_{\beta,1}[\mu,\eta] &:= W^\ast_{\partial\Omega}[\mu] + \nu_{\Omega} \cdot \nabla  v^-_{\tilde{\omega}}[\eta]_{|[0,T] \times\partial\Omega} \\ 
	& \qquad \qquad \qquad \qquad \qquad \qquad \qquad \qquad  \text{on } [0,T] \times \partial\Omega,
    \\
    \mathcal{J}^C_{\beta,2}[\mu,\eta] &:=  W^\ast_{\partial\tilde{\omega}}[\eta]
    +
    \nu_{\tilde{\omega}} \cdot \nabla  v^+_{\Omega}[\mu]_{|[0,T] \times\partial\tilde{\omega}}
    - \beta (v^+_{\Omega}[\mu]_{|[0,T] \times\partial\tilde{\omega}} + V_{\partial\tilde{\omega}}[\eta])  \\ &\qquad \qquad  \qquad \qquad \qquad \qquad  \qquad \qquad\text{on } [0,T] \times \partial\tilde{\omega},
\end{split}
\]
is compact. Since compact perturbations of
linear homeomorphisms are Fredholm operators of index $0$, we have that $\mathcal{J}_{\beta} = \bar{\mathcal{J}} + \mathcal{J}^C_{\beta}$
is a Fredholm operator of index $0$.  Therefore, in order to complete the proof, it suffices to prove that $\mathcal{J}_\beta$ is injective. Thus, we now assume that $(\mu,\eta) \in C_0^{\frac{\alpha}{2};  \alpha}([0,T] \times \partial\Omega) \times C_0^{\frac{\alpha}{2};  \alpha}([0,T] \times \partial\tilde{\omega})$ is such that
\begin{equation}\label{J_beta=0}
\mathcal{J}_{\beta}[\mu,\eta] = (0,0).
\end{equation}

Then, by the jump formulas of Theorem \ref{thmsl} (iii)  and by \eqref{J_beta=0}, we deduce that the function ${u_{\Omega,\tilde{\omega}}}[\mu,\eta]$ defined by \eqref{U} is a solution of the boundary value problem \eqref{beta problem eq}.
Then by Lemma \ref{beta prob lemma}, we have that  ${u_{\Omega,\tilde{\omega}}}[\mu,\eta] = 0$,  which implies $(\mu,\eta)=(0,0)$, by the uniqueness of the representation provided by Lemma \ref{lemma rappr}.
\qed \end{proof}

\section{The perturbed nonlinear mixed problem (\ref{princeqpertu})}\label{sec:pert}

We now take into consideration the perturbed problem \eqref{princeqpertu}. As mentioned, we aim to demonstrate, under suitable conditions, the existence of solutions $u_\phi$ for $\phi$ close to the identity and show that the ``domain-to-solution'' map $\phi\mapsto u_\phi$ is smooth.

The first step is transforming the nonlinear problem \eqref{princeqpertu} into a nonlinear system of integral equations.

Let $\alpha \in \mathopen]0,1[$, $T>0$ and let $\Omega$, $\omega$ be as in \eqref{introsetconditions}. Let $G$, $f$ be as in \eqref{introfunconditions}. {We  note that the assumptions {in \eqref{introfunconditions} on} $G$ imply that $\mathcal{N}_{G}$ maps  ${\mathcal{A}_{\partial \omega}^\Omega\times}C_0^{\frac{1+\alpha}{2};1+\alpha}([0,T] \times \partial\omega)$ to {$C_0^{\frac{\alpha}{2};\alpha}([0,T] \times \partial\omega)$}.}

Then let $\mathcal{M}=(\mathcal{M}_1,\mathcal{M}_2)$ be the map from $\mathcal{A}^{\Omega}_{\partial\omega} \times C_0^{\frac{\alpha}{2};  \alpha}([0,T] \times \partial\Omega) \times C_0^{\frac{\alpha}{2};  \alpha}([0,T] \times \partial\omega)$ to $ C_0^{\frac{\alpha}{2};  \alpha}([0,T] \times \partial\Omega) \times C_0^{\frac{\alpha}{2};  \alpha}([0,T] \times \partial\omega)$ defined by
\begin{equation}\label{M}
\begin{aligned}
\mathcal{M}_1[\phi,\mu,\eta] & := \left(\frac{1}{2} I + W^\ast_{\partial\Omega} \right) [\mu] + \nu_{\Omega} \cdot \nabla  v^-_{\omega[\phi]}[\eta \circ (\phi^T)^{(-1)}] - f,
\\
\mathcal{M}_2[\phi,\mu,\eta] & :=
\left(-\frac{1}{2} I + W^\ast_{\partial\omega[\phi]} \right) [\eta \circ (\phi^T)^{(-1)}] \circ \phi^T \\
& \quad + \left({\nu_{\omega[\phi]}}\circ \phi \right) \cdot \nabla  v^+_{\Omega}[\mu] \circ \phi^T
\\
& \quad
- \mathcal{N}_G\Big({\phi,} v^+_{\Omega}[\mu]_{|[0,T] \times \partial\omega[\phi]}\circ \phi^T + V_{\partial\omega[\phi]}[\eta\circ \phi^{(-1)}] \circ \phi^T \Big),
\end{aligned}
\end{equation}
for all $(\phi,\mu, \eta) \in \mathcal{A}^{\Omega}_{\partial\omega} \times C_0^{\frac{\alpha}{2};  \alpha}([0,T] \times \partial\Omega) \times C_0^{\frac{\alpha}{2};  \alpha}([0,T] \times \partial\omega)$. {Here, we have used the following notation, which will also be adopted throughout the remainder of the paper: If $A$ is a subset of $\mathbb{R}^n$,
$T >0$ and $h$ is a map from $A$ to $\mathbb{R}^n$, we denote by $h^T$ the map from  $[0,T] \times A$
 to  $[0,T] \times \mathbb{R}^n$ defined by 
\[
h^T(t,x) := (t, h(x)) \quad \forall (t,x) \in  [0,T] \times A.
\]
}

  From the definition of $\mathcal{M}$, we readily deduce the following:

\begin{proposition}\label{prop M=0}
Let $\alpha \in \mathopen]0,1[$ and $T>0$. Let $\Omega$, $\omega$ be as in \eqref{introsetconditions}. Let $G$, $f$ be as in \eqref{introfunconditions}. Let
\begin{equation*}
(\phi,\mu, \eta) \in \mathcal{A}^{\Omega}_{\partial\omega} \times C_0^{\frac{\alpha}{2};  \alpha}([0,T] \times \partial\Omega) \times C_0^{\frac{\alpha}{2};  \alpha}([0,T] \times \partial\omega).
\end{equation*}
{Then the function 
\[
u_{\Omega,\omega[\phi]}[\mu,\eta\circ (\phi^T)^{(-1)}]:=(v^+_{\Omega} [\mu] + v^-_{\omega[\phi]}[\eta\circ (\phi^T)^{(-1)}])_{| [0,T] \times (\overline{\Omega} \setminus {\omega[\phi]})}
\] 
(cf.~\eqref{U})}  is a solution of problem \eqref{princeqpertu} if and only if 
\begin{equation}\label{M=0}
\mathcal{M}[\phi,\mu, \eta] = (0,0).
\end{equation}
\end{proposition}
\begin{proof}
By the regularity of $\phi \in \mathcal{A}^{\Omega}_{\partial\omega}$ we have that if $(\mu, \eta) \in C_0^{\frac{\alpha}{2};  \alpha}([0,T] \times \partial\Omega) \times C_0^{\frac{\alpha}{2};  \alpha}([0,T] \times \partial\omega)$, then
    \begin{equation*}(\mu,\eta \circ (\phi^T)^{(-1)}) \in C_0^{\frac{\alpha}{2};  \alpha}([0,T] \times \partial\Omega) \times C_0^{\frac{\alpha}{2};  \alpha}([0,T] \times \phi(\partial\omega)).
    \end{equation*}
    Moreover, since $\overline{\omega[\phi]} \subseteq \Omega$ (cf.~\eqref{A^Omega_omega}), we can apply Lemma \ref{lemma rappr} with $\tilde{\omega}=\omega[\phi]$. Then, by the jump relations for the single layer potential (cf.~Theorem \ref{thmsl} (iii)), by a change of variable on $\phi(\partial \omega)$, and by the definition of $\mathcal{M}$ in \eqref{M}, we deduce that the function
    \begin{equation*}
    u_{\Omega,\omega[\phi]}[\mu, \eta \circ (\phi^T)^{(-1)}] = (v^+_{\Omega} [\mu] + v^-_{\omega[\phi]}[\eta \circ (\phi^T)^{(-1)})_{|[0,T]\times \overline{\Omega} \setminus \omega[\phi]}
    \end{equation*}
    is a solution of problem \eqref{princeqpertu} if and only if \eqref{M=0} is satisfied. 
    \qed \end{proof}

To carry on our analysis, we will assume the existence of a solution $u_0 \in C_{0}^{\frac{1+\alpha}{2}; 1+\alpha}([0,T] \times (\overline{\Omega} \setminus \omega))$ to problem \eqref{princeqpertu} for {$\phi=\phi_0=\mathrm{id}_{\partial \omega}$}. In this paper, we will not explore conditions ensuring the existence of such solution, which can be obtained by several methods. For example, one may adapt the argument from \cite{DaMoMu22} and impose a growth condition on $G$ to apply the Leray-Schauder Fixed Point Theorem {(see also \cite{Mo25})}. For a {discussion on} nonlinear boundary value problems for parabolic equations, readers may consult Friedman \cite[Chapter~7]{Fr08}.

Combining Lemma \ref{lemma rappr} and Proposition \ref{prop M=0}, we can readily see that such solution $u_0$ can be expressed as a sum of single layer potentials. Namely, we have the following:

\begin{proposition}\label{prop u_0}
Let $\alpha \in \mathopen]0,1[$ and $T>0$. Let $\Omega$, $\omega$ be as in \eqref{introsetconditions}. Let $G$, $f$ be as in \eqref{introfunconditions}. Let $\phi_0 := \mathrm{id}_{\partial\omega}$. Assume that $u_0 \in C_{0}^{\frac{1+\alpha}{2}; 1+\alpha}([0,T] \times (\overline{\Omega} \setminus \omega))$ is a solution of
{\[
\begin{cases}
    \partial_t u - \Delta u = 0 & \quad\text{in } ]0,T] \times (\Omega \setminus \overline{\omega}), 
    \\
    \frac{\partial}{\partial \nu_{\Omega}}  u(t,x) = f(t,x) & \quad \forall (t,x)\in [0,T] \times \partial \Omega, 
    \\
        {\frac{\partial}{\partial \nu_{\omega}}  u(t,x)  = G(t,x,u(t,x))} & \quad \forall (t,x)\in [0,T] \times \partial \omega,
    \\
    u(0,\cdot)=0 & \quad \text{in } \overline{\Omega} \setminus \omega\, .
    \end{cases}
\]}
{Then} there exists a unique pair  $(\mu_0,\eta_0) \in C_0^{\frac{\alpha}{2};  \alpha}([0,T] \times \partial\Omega) \times C_0^{\frac{\alpha}{2};  \alpha}([0,T] \times \partial\omega)$  such that
\[
u_0=u_{\Omega,\omega[\phi_0]}[\mu_0,\eta_0\circ (\phi_0^T)^{(-1)}]{=u_{\Omega,\omega}[\mu_0,\eta_0]}\, .
\]
In particular, 
\begin{equation*}
\mathcal{M}[\phi_0,\mu_0,\eta_0] = (0,0)\, .
\end{equation*}
\end{proposition}

According to Proposition \ref{prop M=0}, the analysis of problem \eqref{princeqpertu} is reduced to that of equation \eqref{M=0}, which we intend to study using the Implicit Function Theorem for $C^{\infty}$ maps in Banach spaces (cf.~Deimling \cite[Theorem~15.1 \& Corollary 15.1]{De85}). Consequently, we need to examine the regularity of the map $\mathcal{M}$. If we were dealing with an elliptic equation, we would expect an analyticity result at this point (see, e.g., \cite{La07,LaRo04}). However, this is not the case for the heat equation.

Indeed, in the definition of $\mathcal{M}$, we have the pull-back of integral operators associated with heat layer potentials. {By} the results of \cite{DaLu23} we {only know that these are of class $C^\infty$}. This crucial difference from the elliptic case {obstructs} the proof of analyticity results and reflects the regularity of the fundamental solution $S_n$, which is analytic in $\mathbb{R}^n\setminus\{0\}$ in the case of elliptic operators but only $C^\infty$ on $\mathbb{R}^{1+n}\setminus\{(0,0)\}$ for the heat equation.

In the definition of $\mathcal{M}$ we also have a superposition operator associated with the function $G$. To handle this, we will assume that
{{\begin{equation}\label{condition NG*}
\text{$\mathcal{N}_{G}$ is of class $C^\infty$ from $\mathcal{A}_{\partial \omega}^\Omega \times C^{\frac{1+\alpha}{2};1+\alpha}([0,T] \times \partial\omega)$ to $ C^{\frac{\alpha}{2};\alpha}([0,T] \times \partial\omega)$.}
\end{equation}}
For example, a suitable choice of $G$ {ensuring that $\mathcal{N}_{G}$ is of class $C^\infty$ from $\mathcal{A}_{\partial \omega}^\Omega\times C^{\frac{1+\alpha}{2};1+\alpha}([0,T] \times \partial\omega)$ to $ C^{\frac{\alpha}{2};\alpha}([0,T] \times \partial\omega)$} is any polynomial function $G(t,x,\xi)=a_0(t)b_0(x)+a_1(t)b_1(x)\xi+\ldots+a_k(t)b_k(x)\xi^k$ with $a_0(0)=0$ and where for each $j \in \{1,\dots, k\}$ we have $a_j \in C^{\frac{\alpha}{2}}([0,T])$ and $b_j\in C^\infty(\Omega)$. Furthermore, we note that by assumptions \eqref{introfunconditions} and \eqref{condition NG*} and by computing the differentials of $\mathcal{N}_{G}$ (as it is done for example in Valent \cite[Proof of Thm.~4.1]{Va88}), one verifies that $\mathcal{N}_{G}$ is of class $C^\infty$ also from the product of $\mathcal{A}_{\partial \omega}^\Omega$ with the function space $C^{\frac{1+\alpha}{2};1+\alpha}_0([0,T] \times \partial\omega)$, consisting of functions that are zero at $t=0$, to $C^{\frac{\alpha}{2};\alpha}_0([0,T] \times \partial\omega)$.}


Finally, to study the regularity of the map $\mathcal{M}$, we also need the following technical lemma on the smoothness of certain maps related to the change of variables in integrals and to the {pull-back} of the
outer normal field.  A proof can be found in Lanza de Cristoforis and Rossi \cite[p.~166]{LaRo04} and Lanza de Cristoforis \cite[Prop.~1]{La07}.  

\begin{lemma}\label{lemma change of variable}
	Let $\alpha \in \mathopen]0,1[$. Let $\omega$ be an open bounded connected subset of $\mathbb{R}^n$ of class $C^{1,\alpha}$ with connected exterior $\mathbb{R}^n\setminus \overline{\omega}$. Let $\mathcal{A}_{\partial\omega}$ be as in \eqref{A_omega}. Then the following hold.
    \begin{itemize}
        \item[(i)] For each $\phi \in \mathcal{A}_{\partial\omega}$ there exists a unique $\tilde{\sigma}_n[\phi] \in C^{0,\alpha}(\partial\omega)$ such that 
        \[
    	\int_{\phi(\partial\omega)} f(y) \,d\sigma_y = \int_{\partial\omega} f(\phi(s)) \, \tilde{\sigma}_n[\phi](s) \,d\sigma_s \quad\forall f \in L^1(\phi(\partial\omega)).
    	\]
     Moreover, $\tilde{\sigma}_n[\phi] >0$ and the map taking $\phi$ to $\tilde{\sigma}_n[\phi] $ is real analytic from $\mathcal{A}_{\partial\omega}$ to $C^{0,\alpha}(\partial\omega)$.

     \item[(ii)] The map from  $\mathcal{A}_{\partial\omega}$ to $C^{0,\alpha}(\partial\omega)$  that  takes $\phi$ to ${\nu_{\omega[\phi]}} \circ \phi$ is real analytic.
    \end{itemize}
\end{lemma}

We can now prove that  the map $\mathcal{M}$ is of class $C^\infty$.

\begin{proposition}\label{prop Mrealanal}
  Let $\alpha \in \mathopen]0,1[$ and $T>0$. Let $\Omega$, $\omega$ be as in \eqref{introsetconditions}. Let $G$, $f$ be as in \eqref{introfunconditions}.  Let assumption \eqref{condition NG*} holds. Then the map $\mathcal{M}$ from $\mathcal{A}^{\Omega}_{\partial\omega} \times C_0^{\frac{\alpha}{2};  \alpha}([0,T] \times \partial\Omega) \times C_0^{\frac{\alpha}{2};  \alpha}([0,T] \times \partial\omega)$ to $C_0^{\frac{\alpha}{2};  \alpha}([0,T] \times \partial\Omega) \times C_0^{\frac{\alpha}{2};  \alpha}([0,T] \times \partial\omega)$ is of class $C^\infty$.
\end{proposition}

\begin{proof}
    We confine to analyze $\mathcal{M}_2$. The proof that $\mathcal{M}_1$ is smooth follows a similar but simpler argument, which we leave to the reader.

    We begin by noting that the map from $\mathcal{A}^{\Omega}_{\partial\omega} \times C_0^{\frac{\alpha}{2};  \alpha}([0,T] \times \partial\omega)$ to $C_0^{\frac{\alpha}{2};  \alpha}([0,T] \times \partial\omega)$ that  takes a pair $(\phi,\eta)$ to the function of the variables $(t,x) \in [0,T]\times \partial\omega$ defined by
    \begin{equation*}
    \begin{split}
  &  \left(-\frac{1}{2} I +  W^\ast_{\partial\omega[\phi]} \right) [\eta \circ (\phi^T)^{(-1)}] (\phi^T(t,x)) \\ & \qquad \qquad = -\frac{1}{2}\eta(t,x)  + W^\ast_{\partial\omega[\phi]} [ \eta\circ (\phi^T)^{(-1)}] (\phi^T(t,x)) 
  \end{split}  \end{equation*}
    is of class $C^\infty$, as we can verify by the results on the dependence of heat layer potentials upon perturbation of the support and of the density of \cite[Theorem~5.4]{DaLu23}.
    
    Then we turn to the second term in $\mathcal{M}_2$: The map from $\mathcal{A}^{\Omega}_{\partial\omega} \times C_0^{\frac{\alpha}{2};  \alpha}([0,T] \times \partial\Omega)$ to $C_0^{\frac{\alpha}{2};  \alpha}([0,T] \times \partial\omega)$ which takes $(\phi,\mu)$ to the function of the variables $(t,x) \in [0,T] \times \partial\omega$ defined by 
    \begin{equation*}
    \begin{split}
        {\nu_{\omega[\phi]}} (\phi(x)) \cdot \nabla &  v^+_{\Omega}[\mu] (t,\phi(x)) \\ &= \int_{0}^{t} \int_{\partial \Omega} {\nu_{\omega[\phi]}} (\phi(x)) \cdot \nabla S_{n}(t-\tau,\phi(x)-y) \mu(\tau, y) \,d\sigma_y d\tau 
    \end{split}
    \end{equation*}
    is of class $C^\infty$ by Lemma \ref{lemma change of variable} and by the results of \cite[Lemmas A.2, A.3]{DaLu23} on non-autonomous composition operators and on time-dependent integral operators with non-singular kernels (notice that $\phi(x) - y \neq 0$ for every $(x,y) \in {\partial\omega \times \partial\Omega}$ thanks to the assumption $\phi \in \mathcal{A}^{\Omega}_{\partial\omega}$, cf. \eqref{A^Omega_omega}).

    Next we consider the argument of $\mathcal{N}_G$: The map from $\mathcal{A}^{\Omega}_{\partial\omega} \times C_0^{\frac{\alpha}{2};  \alpha}([0,T] \times \partial\omega)$ to $C_0^{\frac{1+\alpha}{2};1+\alpha}([0,T] \times \partial\omega)$ taking a pair $(\phi,\eta)$ to the function
    \begin{equation*}
    V_{\partial\omega[\phi]}[\eta \circ (\phi^T)^{(-1)}] \circ \phi^T    
    \end{equation*}
    is of class $C^\infty$ {by the results on the dependence of heat layer potentials upon perturbation of the support and of the density of \cite[Theorem~ 5.4]{DaLu23}} and the map from $\mathcal{A}^{\Omega}_{\partial\omega} \times C_0^{\frac{\alpha}{2};  \alpha}([0,T] \times \partial\Omega)$ to $C_0^{\frac{1+\alpha}{2};1+\alpha}([0,T] \times \partial\omega)$ taking a pair $(\phi,\mu)$ to the function of the variables $(t,x) \in [0,T]\times \partial\omega$
    \begin{equation*}
        v^+_{\Omega}[\mu](t,\phi(x)) = \int_{0}^{t} \int_{\partial \Omega} S_{n}(t-\tau, \phi(x)-y) \mu(\tau, y) \,d\sigma_y d\tau
    \end{equation*}
    is of class $C^\infty$ by the results of \cite[Lemmas A.2, A.3]{DaLu23} on non-autonomous composition operators and on time-dependent integral operators with non-singular kernels (notice again that $\phi(x) - y \neq 0$ for every $(x,y) \in {\partial\omega \times \partial\Omega}$, cf. \eqref{A^Omega_omega}). 
    
    Hence, since composition of $C^\infty$ maps {is}
     $C^\infty$, assumption \eqref{condition NG*} ensures that the map $\mathcal{A}^{\Omega}_{\partial\omega} \times C_0^{\frac{\alpha}{2};  \alpha}([0,T] \times \partial\Omega) \times C_0^{\frac{\alpha}{2};  \alpha}([0,T] \times \partial\omega)$ to $C_0^{\frac{\alpha}{2};  \alpha}([0,T] \times \partial\omega)$ that takes a triple $(\phi,\mu,\eta)$ to the function 
    \begin{equation*}
        \mathcal{N}_G\Big({\phi,} v^+_{\Omega}[\mu]_{|[0,T] \times \partial\omega[\phi]} \circ \phi^T + V_{\partial\omega[\phi]}[\eta \circ (\phi^T)^{(-1)}] \circ \phi^T \Big)
    \end{equation*}
    is of class $C^\infty$. Thus, the validity of the statement follows.
\qed \end{proof}

{{By computing the Frech\'et derivative of $\mathcal{N}_G$, we first note that under assumption \eqref{condition NG*} the  partial derivative $\partial_\xi G(t_0,x_0,\xi_0)$ of the function $G$ with respect to {the third variable exists at any point $(t_0,x_0,\xi_0)\in [0,T] \times \Omega\times \mathbb{R}$.} Then, using assumption \eqref{condition NG*} and the standard rules of calculus in Banach spaces, we can derive the following formula for the first-order partial differential of $(\phi,u)\mapsto \mathcal{N}_G(\phi,u)$ with respect to $u$:  At any given pair $(\phi,u) \in \mathcal{A}_{\partial \omega}^\Omega \times C^{\frac{1+\alpha}{2};1+\alpha}([0,T] \times \partial\omega)$, we have 
\[
\partial_u \mathcal{N}_G(\phi,u)\,.\,h = \mathcal{N}_{\partial_\xi G}(\phi,u) \,h \qquad \forall h \in C^{\frac{1+\alpha}{2};1+\alpha}([0,T] \times \partial\omega),
\]
where $\partial_\xi G$ denotes the partial derivative of the function $G$ with respect to {the third variable {(cf.~Valent \cite[Proof of Thm.~4.1]{Va88}).} We also note that $\mathcal{N}_{\partial_\xi G}(\phi,u)$ belongs to $C^{\frac{\alpha}{2};  \alpha}([0,T] \times \partial\omega)$ for all $(\phi,u) \in \mathcal{A}_{\partial \omega}^\Omega \times C^{\frac{1+\alpha}{2};1+\alpha}([0,T] \times \partial\omega)$.}}}



We can now prove the following proposition, where we show that the partial differential of $\mathcal{M}$ with respect to the densities $(\mu,\eta)$ is invertible at $(\phi_0,\mu_0,\eta_0)$. This constitutes an essential step to apply the  Implicit Function Theorem to equation \eqref{M=0}.

\begin{proposition}\label{d M prop}
Let $\alpha \in \mathopen]0,1[$ and $T>0$. Let $\Omega$, $\omega$ be as in \eqref{introsetconditions}. Let $G$, $f$ be as in \eqref{introfunconditions}. Let assumption \eqref{condition NG*} hold. Let $\phi_0$, $u_0$, $(\mu_0,\eta_0)$  be as in Proposition \ref{prop u_0}. Then the partial differential of $\mathcal{M}$ with respect to $(\mu,\eta)$ evaluated at the point $(\phi_0,\mu_0,\eta_0)$, which we denote by 
\begin{equation}\label{partdiff M}
\partial_{(\mu,\eta)} \mathcal{M}[\phi_0,\mu_0, \eta_0],
\end{equation}
is an homeomorphism from $C_0^{\frac{\alpha}{2};  \alpha}([0,T] \times \partial\Omega) \times C_0^{\frac{\alpha}{2};  \alpha}([0,T] \times \partial\omega)$ into itself.
\end{proposition}

\begin{proof}
By standard calculus in Banach spaces, we can verify that the partial differential \eqref{partdiff M} is the linear and continuous operator from $C_0^{\frac{\alpha}{2};  \alpha}([0,T] \times \partial\Omega) \times C_0^{\frac{\alpha}{2};  \alpha}([0,T] \times \partial\omega)$ {to} itself {given} by
\begin{equation}\label{d M eq}
\begin{aligned}
&\partial_{(\mu,\eta)} \mathcal{M}_1[\phi_0,\mu_0, \eta_0]\,.\,(\mu,\eta)  {=} \left(\frac{1}{2} I + W^\ast_{\partial\Omega} \right) [\mu] + {\nu_{\Omega} \cdot \nabla  v^-_{\omega}[\eta]},
\\
&\partial_{(\mu,\eta)} \mathcal{M}_2[\phi_0,\mu_0, \eta_0]\,.\,(\mu,\eta)  {=}
\left(-\frac{1}{2} I + W^\ast_{\partial\omega} \right) [\eta] + {\nu_{\omega}} \cdot \nabla  v^+_{\Omega}[\mu]
\\
& \quad
- \mathcal{N}_{\partial_\xi G}\Big({\phi_0,}v^+_{\Omega}[\mu_0]_{|[0,T]\times\partial\omega} +  {V_{\partial\omega}}[\eta_0]\Big) (v^+_{\Omega}[\mu]_{|[0,T]\times\partial\omega} +  {V_{\partial\omega}}[\eta]),
\end{aligned}
\end{equation}
for all $(\mu,\eta) \in C_0^{\frac{\alpha}{2};  \alpha}([0,T] \times \partial\Omega) \times C_0^{\frac{\alpha}{2};  \alpha}([0,T] \times \partial\omega)$. Then, comparing \eqref{d M eq} with \eqref{J_beta eq}, Proposition \ref{prop J_beta} with the choice
\begin{equation*}
    \beta := \mathcal{N}_{\partial_\xi G}({\phi_0,}v^+_{\Omega}[\mu_0]_{|[0,T]\times\partial\omega} +  {V_{\partial\omega}}[\eta_0])
\end{equation*}
implies the validity of the statement. Notice that $\beta\in C^{\frac{\alpha}{2};  \alpha}([0,T] \times \partial {\omega})$ by assumption \eqref{condition NG*}.
\qed \end{proof}

With Propositions \ref{prop Mrealanal} and   \ref{d M prop} we have all the necessary ingredients to apply the Implicit Function Theorem {(cf.~Deimling \cite[Theorem~15.1 \& Corollary 15.1]{De85})} to equation \eqref{M=0}. We obtain the following:

\begin{theorem}\label{Lambda Thm}
  Let $\alpha \in \mathopen]0,1[$ and $T>0$. Let $\Omega$, $\omega$ be as in \eqref{introsetconditions}. Let $G$, $f$ be as in \eqref{introfunconditions}.  Let assumption  \eqref{condition NG*} hold. Let $\phi_0$, $u_0$, $(\mu_0,\eta_0)$  be as in Proposition \ref{prop u_0}. Then, there exist two open neighborhoods $Q_0$ of $\phi_0$ in $\mathcal{A}^{\Omega}_{\partial\omega}$ and $H_0$ of $(\mu_0,\eta_0)$ in $C_0^{\frac{\alpha}{2};  \alpha}([0,T] \times \partial\Omega) \times C_0^{\frac{\alpha}{2};  \alpha}([0,T] \times \partial\omega)$, and a $C^\infty$ map 
    \begin{equation*}
        \Lambda := (\Lambda_1,\Lambda_2): Q_0 \to H_0
    \end{equation*}
    such that the set of zeros of $\mathcal{M}$ in $Q_0 \times H_0$ coincides with the graph of the function $\Lambda$.  In particular,
   { \begin{equation*}
\mathcal{M}[ \phi, \Lambda_1[\phi],\Lambda_2[\phi]]=0 \quad \forall \phi \in Q_0\, ,\qquad   \Lambda[\phi_0]= 
    (\Lambda_1[\phi_0],\Lambda_2[\phi_0])= (\mu_0,\eta_0).
    \end{equation*}}
\end{theorem}

Theorem \ref{Lambda Thm} yields a family of solutions smoothly depending on $\phi$ for the system of integral equations in \eqref{M=0}, which in turn can be used to generate a family of solutions for the equivalent perturbed boundary value problem \eqref{princeqpertu} (cf.~Proposition \ref{prop M=0}, see also {Proposition} \ref{prop u_0} and Theorem \ref{Lambda Thm}).

\begin{theorem}\label{u_phi thm}
Let $\alpha \in \mathopen]0,1[$ and $T>0$. Let $\Omega$, $\omega$ be as in \eqref{introsetconditions}. Let $G$, $f$ be as in \eqref{introfunconditions}.  Let assumption  \eqref{condition NG*} hold. Let $\phi_0$, $u_0$, $(\mu_0,\eta_0)$  be as in Proposition \ref{prop u_0}.  Let $Q_0$ and $\Lambda := (\Lambda_1,\Lambda_2)$ be as in Theorem \ref{Lambda Thm}. For each $\phi \in Q_0$, let
{\begin{equation*}
\begin{split}
u_{\phi} &:= u_{\Omega,\omega[\phi]}[\Lambda_1[\phi], \Lambda_2[\phi] \circ (\phi^T)^{(-1)}]\\
&=(v^+_{\Omega} [\Lambda_1[\phi]] + v^-_{\omega[\phi]}[\Lambda_2[\phi]\circ (\phi^T)^{(-1)}])_{| [0,T] \times (\overline{\Omega} \setminus {\omega[\phi]})}
\end{split}
\end{equation*}
(cf.~\eqref{U}).} Then $u_{\phi}$ is a solution of  \eqref{princeqpertu} and $u_{\phi_0}= u_0$.
\end{theorem}

In the following theorem, {we take a bounded open set $\Omega_\mathtt{int}$  such that
\[
{\overline \Omega_\mathtt{int}}\subseteq  \Omega\setminus \overline{\omega}
\] 
(see Figure \ref{fig:2}) and for $\phi$ close to the identity we show that restriction of the solution $u_\phi$ to ${\overline \Omega_\mathtt{int}}$  depends smoothly on  $\phi$.}

\begin{figure}[!htb]
\centering
\includegraphics[width=3.2in]{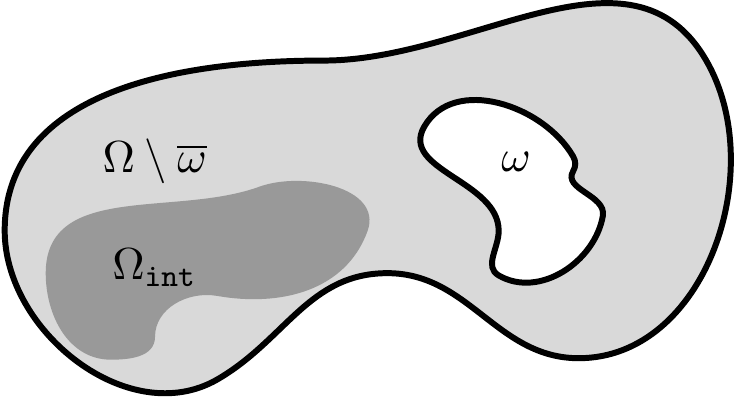}
\caption{{\it The sets $\omega$, $\Omega\setminus \overline{\omega}$, and $\Omega_\mathtt{int}$.}}\label{fig:2} 
\end{figure}

\begin{theorem}\label{thm:smoothrep}
  Let $\alpha \in \mathopen]0,1[$ and $T>0$. Let $\Omega$, $\omega$ be as in \eqref{introsetconditions}. Let $G$, $f$ be as in \eqref{introfunconditions}. Let assumption  \eqref{condition NG*} hold. Let $\phi_0$, $u_0$, $(\mu_0,\eta_0)$  be as in Proposition \ref{prop u_0}. Let $Q_0$ and $\Lambda := (\Lambda_1,\Lambda_2)$ be as in Theorem \ref{Lambda Thm} and let {$\{u_\phi\}_{\phi \in Q_0}$} be as in Theorem  \ref{u_phi thm}.
{Let $\Omega_\mathtt{int}$ be a bounded open subset of $\Omega\setminus \overline{\omega}$ such that
\[
{\overline \Omega_\mathtt{int}}\subseteq  \Omega\setminus \overline{\omega}
\]} 
and let $Q_\mathtt{int} \subseteq Q_0$ be an open neighborhood of $\phi_0$ such that  
\begin{equation}\label{cond Omega int}
    \overline{\Omega_\mathtt{int}} \subseteq \Omega \setminus \overline{\omega[\phi]} \quad \forall \phi \in Q_\mathtt{int}.
\end{equation}
Then the map from $Q_\mathtt{int}$ to $C_{0}^{\frac{1+\alpha}{2}; 1+\alpha}([0,T] \times \overline{\Omega_\mathtt{int}})$ that takes $\phi$ to $\left(u_{\phi}\right)_{|[0,T]\times\overline{\Omega_\mathtt{int}}}$ is of class $C^\infty$.
\end{theorem}

\begin{proof}
    {Without loss of generality we can assume that $\Omega_\mathtt{int}$ is of class $C^{1,\alpha}$.} By Theorem \ref{u_phi thm}, definition \eqref{U}, and Lemma \ref{lemma change of variable}, for every $\phi \in Q_\mathtt{int}$ we have that
    \begin{equation}\label{eq:rep}
    \begin{split}
        u_{\phi}(t,x) &= u_{\Omega,\omega[\phi]}[\Lambda_1[\phi], \Lambda_2[\phi] \circ (\phi^T)^{(-1)}] (t,x)
        \\
       & =\int_{0}^{t} \int_{\partial \Omega} S_{n}(t-\tau, x-y) \Lambda_1[\phi](\tau, y) \,d\sigma_y d\tau 
        \\ 
      &  \qquad +
        \int_{0}^{t} \int_{\phi(\partial \omega)} S_{n}(t-\tau, x-\tilde{y}) \Lambda_2[\phi] \circ (\phi^T)^{(-1)} (t,\tilde{y})\,d\sigma_{\tilde{y}} d\tau
        \\
        &= \int_{0}^{t} \int_{\partial \Omega} S_{n}(t-\tau, x-y) \Lambda_1[\phi](\tau, y) \,d\sigma_y d\tau 
        \\
       & \qquad  +\int_{0}^{t} \int_{\partial \omega} S_{n}(t-\tau, x-\phi(y)) \Lambda_2[\phi](\tau, y) \tilde{\sigma}_n[\phi](y) \,d\sigma_y d\tau \, ,
    \end{split}
    \end{equation}
    for all $(t,x) \in [0,T] \times \overline{\Omega_\mathtt{int}}$. Moreover, by \eqref{cond Omega int} we notice that
    \begin{equation*}
        x-y \neq 0 \quad\forall (x,y) \in  \overline{\Omega_\mathtt{int}} \times \partial \Omega \quad \text{and} \quad x-\phi(y) \neq 0 \quad\forall (x,y) \in  \overline{\Omega_\mathtt{int}} \times \partial\omega.
    \end{equation*}
    {
 Then, by \cite[Lemmas  A.2, A.3]{DaLu23} on the regularity of time-dependent integral operators with non-singular kernels and of superposition operators, by Theorem \ref{Lambda Thm}, and by Lemma \ref{lemma change of variable}, we deduce that the integral operators in {the right hand side of} \eqref{eq:rep} define $C^\infty$ maps from $Q_\mathtt{int}$ to $C_{0}^{\frac{1+\alpha}{2}; 1+\alpha}([0,T] \times \overline{\Omega_\mathtt{int}})$ (see also the proof of \cite[Theorem~5.6]{DaLuMoMu24}). The validity of the statement follows.}
\qed \end{proof}

{

Since the family of solutions $\{u_\phi\}_{\phi \in Q_0}$ is built through the Implicit Function Theorem, in Proposition \ref{prop:locrep} below, we deduce the validity of a certain local uniqueness property. In particular, we will prove that   if $\tilde{\phi} \in Q_0$, $\tilde{u}$ is a solution of problem \eqref{princeqpertu} with $\phi$ replaced with $\tilde{\phi}$, and the trace of $\tilde{u}$ on $[0,T]\times \partial (\Omega \setminus \overline{\omega[\tilde{\phi}]})$ belongs to a certain set, then  $\tilde{u}=u_{\tilde{\phi}}$, i.e., $\tilde{u}$ must be the element of the family $\{u_\phi\}_{\phi \in Q_0}$ with $\phi=\tilde{\phi}$.

{\begin{proposition}\label{prop:locrep}
  Let $\alpha \in \mathopen]0,1[$ and $T>0$. Let $\Omega$, $\omega$ be as in \eqref{introsetconditions}. Let $G$, $f$ be as in \eqref{introfunconditions}. Let assumption  \eqref{condition NG*} hold. Let $\phi_0$, $u_0$ be as in Proposition \ref{prop u_0} and $Q_0$ be the neighborhood of $\phi_0$ of Theorem \ref{Lambda Thm}. Then there exists an open neighborhood $U_0$ of $(u_{0|[0,T] \times\partial\Omega},u_{0|[0,T] \times\partial\omega})$ in $C_0^{\frac{1+\alpha}{2}; 1+\alpha}([0,T] \times \partial\Omega)\times C_0^{\frac{1+\alpha}{2}; 1+\alpha}([0,T] \times \partial\omega)$ such that the following statement holds: 
  
 If $\tilde{\phi} \in Q_0$, $\tilde{u}$ is a solution of problem \eqref{princeqpertu} with $\phi$ replaced with $\tilde{\phi}$, and 
 \[
(\tilde{u}_{|[0,T] \times\partial\Omega},\tilde{u}\circ\tilde\phi^T)\in U_0\,,
 \] 
  then 
\[
  \tilde{u}=u_{\tilde{\phi}}\,,
\]  
where $\{u_\phi\}_{\phi \in Q_0}$ is as in Theorem  \ref{u_phi thm}.
 \end{proposition}
 \begin{proof}
 Let
  {\small \begin{equation*}
  \begin{split}
      U_0:=\Bigg\{ \bigg((v^+_{\Omega} [\mu] + v^-_{\omega[\phi]}[\eta\circ (\phi^T)^{(-1)}])_{| [0,T] \times \partial \Omega}, v^+_{\Omega} [\mu]\circ \phi^T + v^-_{\omega[\phi]}[\eta \circ (\phi^T)^{(-1)}]\circ \phi^T  \bigg)\colon
      \\
       (\phi,\mu,\eta)\in Q_0\times H_0  \Bigg\}
  \end{split}
  \end{equation*}}
 with $H_0$ as in Theorem \ref{Lambda Thm}. By Theorem \ref{thmsl} (iv), $U_0$ is an open neighborhood of $(u_{0|[0,T] \times\partial\Omega},u_{0|[0,T] \times\partial\omega})$ in $C_0^{\frac{1+\alpha}{2}; 1+\alpha}([0,T] \times \partial\Omega)\times C_0^{\frac{1+\alpha}{2}; 1+\alpha}([0,T] \times \partial\omega)$. Moreover, Theorem \ref{thmsl} (iv) also implies that there exists a unique pair $(\tilde{\mu},\tilde{\eta}) \in C_0^{\frac{\alpha}{2};  \alpha}([0,T] \times \partial\Omega) \times C_0^{\frac{\alpha}{2};  \alpha}([0,T] \times \partial\omega)$ such that
{\small \begin{equation*}
     \begin{split}
         (&\tilde{u}_{|[0,T]\times \partial \Omega}, \,\tilde{u}\circ \tilde{\phi}^T)
         \\
         & = \bigg((v^+_{\Omega} [\tilde{\mu}] + v^-_{\omega[\tilde{\phi}]}[\tilde{\eta}\circ (\tilde{\phi}^T)^{(-1)}])_{| [0,T] \times \partial \Omega}, v^+_{\Omega} [\tilde{\mu}]\circ \tilde{\phi}^T + v^-_{\omega[\tilde{\phi}]}[\tilde{\eta} \circ (\tilde{\phi}^T)^{(-1)}]\circ \tilde{\phi}^T \bigg)\, .
     \end{split}
 \end{equation*}}
 Since $(\tilde{u}_{|[0,T]\times \partial \Omega}, \tilde{u}\circ \tilde{\phi}^T) \in U_0$, then $(\tilde{\mu},\tilde{\eta}) \in H_0$. Since $\tilde{u}$ is a solution of problem \eqref{princeqpertu} with $\phi$ replaced with $\tilde{\phi}$, then
 \[
 \mathcal{M}[\tilde{\phi},\tilde{\mu},\tilde{\eta}]=0\, .
 \]
  As a consequence, since $(\tilde{\phi},\tilde{\mu},\tilde{\eta}) \in Q_0\times H_0$, by Theorem \ref{Lambda Thm} we have
 \[
 (\tilde{\mu},\tilde{\eta})=(\Lambda_1[\tilde{\phi}],\Lambda_2[\tilde{\phi}])\, ,
 \]
 which implies that $\tilde{u}=u_{\tilde{\phi}}$.
 \qed \end{proof}
 
}}

{
\begin{remark}
We observe that under suitable assumptions on the family of functions $\{G_\phi\}_{\phi \in \mathcal{A}_{\partial \omega}^\Omega}$ we could consider {nonlinear} mixed  boundary value problems of the type
{\[
\begin{cases}
    \partial_t u - \Delta u = 0 & \quad\text{in } ]0,T] \times (\Omega \setminus \overline{\omega[\phi]}), 
    \\
    \frac{\partial}{\partial \nu_{\Omega}}  u(t,x) = f (t,x)& \quad \forall (t,x)\in [0,T] \times \partial \Omega, 
    \\
    \frac{\partial}{\partial {\nu_{\omega[\phi]}}}  u (t,x) = G_\phi(t,x,u(t,x)) & \quad \forall (t,x)\in  [0,T] \times \partial \omega[\phi],
    \\
    u(0,\cdot)=0 & \quad \text{in } \overline{\Omega} \setminus \omega[\phi].
    \end{cases}
\]
However, for the sake of clarity, we decided to keep a fixed function $G$.}
\end{remark}
}

{\section{Some remarks on {a linear problem}}\label{sec:lin}

The results that we have obtained for the study of the nonlinear problem clearly hold also for { some linear cases, where the analysis can be simplified} and where we can easily consider the dependence also on the Neumann datum. 

So let $\alpha \in \mathopen]0,1[$, $T>0$ and let $\Omega$, $\omega$ be as in \eqref{introsetconditions}. Then,  for all $(\phi,\gamma, f) \in \mathcal{A}^{\Omega}_{\partial\omega} \times C^{\frac{\alpha}{2}; \alpha}([0,T] \times \partial\omega) \times C_{0}^{\frac{\alpha}{2}; \alpha}([0,T] \times \partial\Omega)$, we  may consider {for example} the following linear mixed  boundary value problem for a function $u \in C_{0}^{\frac{1+\alpha}{2}; 1+\alpha}([0,T] \times (\overline{\Omega} \setminus \omega[\phi]))$:
{\begin{equation}\label{eq:lin}
\begin{cases}
    \partial_t u - \Delta u = 0 & \quad\text{in } ]0,T] \times (\Omega \setminus \overline{\omega[\phi]}), 
    \\
    \frac{\partial}{\partial \nu_{\Omega}}  u(t,x) = f(t,x) & \quad \forall (t,x)\in [0,T] \times \partial \Omega, 
    \\
    \frac{\partial}{\partial {\nu_{\omega[\phi]}}}  u(t,x)  +\gamma(t,\phi^{(-1)}(x))u(t,x)=0& \quad \forall (t,x)\in [0,T] \times \partial \omega[\phi],
    \\
    u(0,\cdot)=0 & \quad \text{in } \overline{\Omega} \setminus \omega[\phi]\, .
    \end{cases}
\end{equation}}
As is well known, for each $(\phi,\gamma, f) \in \mathcal{A}^{\Omega}_{\partial\omega} \times C^{\frac{\alpha}{2}; \alpha}([0,T] \times \partial\omega) \times C_{0}^{\frac{\alpha}{2}; \alpha}([0,T] \times \partial\Omega)$ problem \eqref{eq:lin} has a unique solution in $C_{0}^{\frac{1+\alpha}{2}; 1+\alpha}([0,T] \times (\overline{\Omega} \setminus \omega[\phi]))$, which we denote by $u[\phi,\gamma,f]$. { We note that we could handle certain Robin conditions of the type
\[
 \frac{\partial}{\partial {\nu_{\omega[\phi]}}}  u(t,x)  +\tilde{\gamma}(t,x)u(t,x)=0 \qquad \forall (t,x)\in [0,T] \times \partial \omega[\phi],
\]
instead of the one in problem \eqref{eq:lin}, where $\tilde{\gamma}$ is some suitable function defined on $[0,T] \times  \Omega$. Here, however, since in this case we also aim at studying the dependence on the Robin parameter, we decided to consider the Robin condition as it is  in \eqref{eq:lin}. For some comments on a different formulation of the Robin condition, see Remark \ref{rem:linbis}.
}

For each $(\phi,\gamma) \in \mathcal{A}^{\Omega}_{\partial\omega} \times C^{\frac{\alpha}{2}; \alpha}([0,T] \times \partial\omega)$, one might be interested in the Neumann-to-Dirichlet operator $\mathrm{NtD}_{(\phi,\gamma)}$, which maps the Neumann datum $f \in C_{0}^{\frac{\alpha}{2}; \alpha}([0,T] \times \partial\Omega)$ to the trace $u[\phi,\gamma,f]_{|[0,T]\times \partial \Omega} \in  C_{0}^{\frac{1+\alpha}{2}; 1+\alpha}([0,T] \times  \partial {\Omega})$ of the solution on $[0,T]\times \partial \Omega$. The map $\mathrm{NtD}_{(\phi,\gamma)}$ is a linear and continuous operator from $C_{0}^{\frac{\alpha}{2}; \alpha}([0,T] \times \partial\Omega)$ to $C_{0}^{\frac{1+\alpha}{2}; 1+\alpha}([0,T] \times \partial\Omega)$, i.e., $\mathrm{NtD}_{(\phi,\gamma)}\in \mathcal{L}(C_{0}^{\frac{\alpha}{2}; \alpha}([0,T] \times \partial\Omega), C_{0}^{\frac{1+\alpha}{2}; 1+\alpha}([0,T] \times \partial\Omega))$. Such an operator can be useful in applications. For instance, in Nakamura and Wang \cite{NaWa15,NaWa17}, the Neumann-to-Dirichlet operator is employed for the reconstruction of an unknown cavity with Robin boundary conditions inside a heat conductor. Here, in particular, we aim to understand the regularity of the map 
\[
\begin{split}
 (\phi,\gamma) \mapsto  \mathrm{NtD}_{(\phi,\gamma)}
\end{split}\]
from $\mathcal{A}^{\Omega}_{\partial\omega} \times C^{\frac{\alpha}{2}; \alpha}([0,T] \times \partial\omega)$ to $\mathcal{L}(C_{0}^{\frac{\alpha}{2}; \alpha}([0,T] \times \partial\Omega), C_{0}^{\frac{1+\alpha}{2}; 1+\alpha}([0,T] \times \partial\Omega))$. To achieve our objective, we can follow the approach used in the nonlinear case.

So, let $\alpha \in \mathopen]0,1[$, $T>0$ and let $\Omega$, $\omega$ be as in \eqref{introsetconditions}. Let $\mathcal{M}=(\mathcal{M}_1,\mathcal{M}_2)$ be the map from $\mathcal{A}^{\Omega}_{\partial\omega} \times C^{\frac{\alpha}{2}; \alpha}([0,T] \times \partial\omega) \times C_{0}^{\frac{\alpha}{2}; \alpha}([0,T] \times \partial\Omega) \times C_0^{\frac{\alpha}{2};  \alpha}([0,T] \times \partial\Omega) \times C_0^{\frac{\alpha}{2};  \alpha}([0,T] \times \partial\omega)$ to $ C_0^{\frac{\alpha}{2};  \alpha}([0,T] \times \partial\Omega) \times C_0^{\frac{\alpha}{2};  \alpha}([0,T] \times \partial\omega)$ defined by
\[
\begin{aligned}
\mathcal{M}_1[\phi,\gamma, f, \mu,\eta] & := \left(\frac{1}{2} I + W^\ast_{\partial\Omega} \right) [\mu] + \nu_{\Omega} \cdot \nabla  v^-_{\omega[\phi]}[\eta \circ (\phi^T)^{(-1)}] - f,
\\
\mathcal{M}_2[\phi,\gamma, f, \mu,\eta] & :=
\left(-\frac{1}{2} I + W^\ast_{\partial\omega[\phi]} \right) [\eta \circ (\phi^T)^{(-1)}] \circ \phi^T 
\\
& \quad
+ \left({\nu_{\omega[\phi]}}\circ \phi \right) \cdot \nabla  v^+_{\Omega}[\mu] \circ \phi^T
\\
& \quad
+ \gamma{(\cdot)} \Big(v^+_{\Omega}[\mu]_{|[0,T] \times \partial\omega[\phi]}\circ \phi^T + V_{\partial\omega[\phi]}[\eta\circ \phi^{(-1)}] \circ \phi^T \Big),
\end{aligned}
\]
for all $(\phi,\gamma, f, \mu, \eta) \in \mathcal{A}^{\Omega}_{\partial\omega} \times C^{\frac{\alpha}{2}; \alpha}([0,T] \times \partial\omega) \times C_{0}^{\frac{\alpha}{2}; \alpha}([0,T] \times \partial\Omega) \times C_0^{\frac{\alpha}{2};  \alpha}([0,T] \times \partial\Omega) \times C_0^{\frac{\alpha}{2};  \alpha}([0,T] \times \partial\omega)$.  From the definition of $\mathcal{M}$, we readily deduce the following:

\begin{proposition}\label{prop Mlin=0}
Let $\alpha \in \mathopen]0,1[$ and $T>0$. Let $\Omega$, $\omega$ be as in \eqref{introsetconditions}. Let
\begin{equation*}
(\phi,\gamma, f) \in \mathcal{A}^{\Omega}_{\partial\omega} \times C^{\frac{\alpha}{2}; \alpha}([0,T] \times \partial\omega) \times C_{0}^{\frac{\alpha}{2}; \alpha}([0,T] \times \partial\Omega) \, .
\end{equation*}
Then 
\[
u[\phi,\gamma,f]=(v^+_{\Omega} [\mu] + v^-_{\omega[\phi]}[\eta\circ (\phi^T)^{(-1)}])_{| [0,T] \times (\overline{\Omega} \setminus {\omega[\phi]})}\, ,
\] 
where $(\mu, \eta)$ is the unique solution in $C_0^{\frac{\alpha}{2};  \alpha}([0,T] \times \partial\Omega) \times C_0^{\frac{\alpha}{2};  \alpha}([0,T] \times \partial\omega)$ of equation
\begin{equation}\label{Mlin=0}
\mathcal{M}[\phi,\gamma, f, \mu, \eta] = (0,0).
\end{equation}
\end{proposition}
Since the solutions of equation \eqref{Mlin=0} play a relevant role, for each $(\phi,\gamma, f) \in \mathcal{A}^{\Omega}_{\partial\omega} \times C^{\frac{\alpha}{2}; \alpha}([0,T] \times \partial\omega) \times C_{0}^{\frac{\alpha}{2}; \alpha}([0,T] \times \partial\Omega)$, we denote by
\begin{equation*}
(\mu[\phi,\gamma, f], \eta[\phi,\gamma, f])    
\end{equation*}
the unique solution in $C_0^{\frac{\alpha}{2};  \alpha}([0,T] \times \partial\Omega) \times C_0^{\frac{\alpha}{2};  \alpha}([0,T] \times \partial\omega)$ of equation \eqref{Mlin=0}. In view of Proposition \ref{prop Mlin=0}, to study the regularity of 
\[
(\phi,\gamma,f) \mapsto u[\phi,\gamma,f]\, ,
\]
 we can start by understanding the regularity of 
 \[
 (\phi,\gamma,f)\mapsto (\mu[\phi,\gamma, f], \eta[\phi,\gamma, f])\, .
 \] 
 We do so in the following theorem, which can be proved by simplifying the arguments employed in Proposition \ref{d M prop} and Theorem \ref{Lambda Thm}. 

\begin{theorem}\label{LambdaThmlin}
  Let $\alpha \in \mathopen]0,1[$ and $T>0$. Let $\Omega$, $\omega$ be as in \eqref{introsetconditions}. Then the map $(\mu[\cdot,\cdot, \cdot], \eta[\cdot,\cdot, \cdot])$
  from $\mathcal{A}^{\Omega}_{\partial\omega} \times C^{\frac{\alpha}{2}; \alpha}([0,T] \times \partial\omega) \times C_{0}^{\frac{\alpha}{2}; \alpha}([0,T] \times \partial\Omega)$ to $C_0^{\frac{\alpha}{2};  \alpha}([0,T] \times \partial\Omega) \times C_0^{\frac{\alpha}{2};  \alpha}([0,T] \times \partial\omega)$
  which takes the triple $(\phi,\gamma, f)$ to {the} unique solution  $(\mu[\phi,\gamma, f], \eta[\phi,\gamma, f])$ of equation \eqref{Mlin=0} in  $C_0^{\frac{\alpha}{2};  \alpha}([0,T] \times \partial\Omega) \times C_0^{\frac{\alpha}{2};  \alpha}([0,T] \times \partial\omega)$ is of class $C^\infty$.
\end{theorem}

Then, by the representation formula of Proposition \ref{prop Mlin=0} and the smoothness result of Theorem \ref{LambdaThmlin}, we can prove that the trace of $u[\phi,\gamma,f]$ on $[0,T]\times \partial \Omega$ depends smoothly upon the triple $(\phi,\gamma,f)$.

\begin{theorem}\label{thm:smoothreplin}
  Let $\alpha \in \mathopen]0,1[$ and $T>0$. Let $\Omega$, $\omega$ be as in \eqref{introsetconditions}. Then the map from $\mathcal{A}^{\Omega}_{\partial\omega} \times C^{\frac{\alpha}{2}; \alpha}([0,T] \times \partial\omega) \times C_{0}^{\frac{\alpha}{2}; \alpha}([0,T] \times \partial\Omega)$ to $C_{0}^{\frac{1+\alpha}{2}; 1+\alpha}([0,T] \times  \partial {\Omega})$ that takes $(\phi,\gamma,f)$ to $u[\phi,\gamma,f]_{|[0,T]\times \partial \Omega}$ is of class $C^\infty$.
\end{theorem}

\begin{proof}
By Proposition \ref{prop Mlin=0} and Theorem \ref{LambdaThmlin}, for every $(\phi,\gamma,f) \in \mathcal{A}^{\Omega}_{\partial\omega} \times C^{\frac{\alpha}{2}; \alpha}([0,T] \times \partial\omega) \times C_{0}^{\frac{\alpha}{2}; \alpha}([0,T] \times \partial\Omega)$ we have that
    \begin{equation}\label{eq:replin}
    \begin{split}
        u[\phi,\gamma,f&]_{|[0,T]\times \partial \Omega}(t,x)
        = \int_{0}^{t} \int_{\partial \Omega} S_{n}(t-\tau, x-y) \mu[\phi,\gamma,f](\tau, y) \,d\sigma_y d\tau 
        \\
        &+\int_{0}^{t} \int_{\partial \omega} S_{n}(t-\tau, x-\phi(y)) \eta[\phi,\gamma,f](\tau, y) \tilde{\sigma}_n[\phi](y) \,d\sigma_y d\tau \, ,
    \end{split}
    \end{equation}
    for all $(t,x) \in [0,T] \times \partial \Omega$. Then, by \cite[Lemmas  A.2, A.3]{DaLu23} on the regularity of time-dependent integral operators with non-singular kernels and of superposition operators, by Theorems \ref{thmsl} and \ref{LambdaThmlin}, and by Lemma \ref{lemma change of variable}, we deduce that the integral operators in the right hand side of \eqref{eq:replin} define $C^\infty$ maps from $\mathcal{A}^{\Omega}_{\partial\omega} \times C^{\frac{\alpha}{2}; \alpha}([0,T] \times \partial\omega) \times C_{0}^{\frac{\alpha}{2}; \alpha}([0,T] \times \partial\Omega)$ to $C_{0}^{\frac{1+\alpha}{2}; 1+\alpha}([0,T] \times  \partial {\Omega})$. The validity of the statement follows.
\qed \end{proof}

Then, by Theorem \ref{thm:smoothreplin}, we deduce that the function $(\phi,\gamma) \mapsto \mathrm{NtD}_{(\phi,\gamma)}$ is of class $C^\infty$.

\begin{corollary}\label{cor:NtD}
  Let $\alpha \in \mathopen]0,1[$ and $T>0$. Let $\Omega$, $\omega$ be as in \eqref{introsetconditions}. Then the map from $\mathcal{A}^{\Omega}_{\partial\omega} \times C^{\frac{\alpha}{2}; \alpha}([0,T] \times \partial\omega)$ to $\mathcal{L}(C_{0}^{\frac{\alpha}{2}; \alpha}([0,T] \times \partial\Omega),C_{0}^{\frac{1+\alpha}{2}; 1+\alpha}([0,T] \times  \partial {\Omega}))$ that takes $(\phi,\gamma)$ to $\mathrm{NtD}_{(\phi,\gamma)}$ is of class $C^\infty$.
\end{corollary}
\begin{proof}
By Theorem \ref{thm:smoothreplin}, if $f_0 \in C_{0}^{\frac{\alpha}{2}; \alpha}([0,T] \times \partial\Omega)$ is fixed, then the map which takes the pair $(\phi,\gamma)$  to the partial differential 
\begin{equation*}
    \partial_{f} u[\phi,\gamma,f_0]_{|[0,T]\times \partial \Omega} \in \mathcal{L}(C_{0}^{\frac{\alpha}{2}; \alpha}([0,T] \times \partial\Omega),C_{0}^{\frac{1+\alpha}{2}; 1+\alpha}([0,T] \times  \partial {\Omega}))
\end{equation*}
is of class $C^\infty$. Since $u[\phi,\gamma,f]$ is linear in $f$, we deduce that
\[
\partial_{f} u[\phi,\gamma,f_0]_{|[0,T]\times \partial \Omega}=u[\phi,\gamma,\cdot]_{|[0,T]\times \partial \Omega}\, ,
\]
and thus the validity of the statement follows.
\qed \end{proof}
}

{
\begin{remark}\label{rem:linbis}
Let $\alpha \in \mathopen]0,1[$, $T>0$ and let $\Omega$, $\omega$ be as in \eqref{introsetconditions}. Let $\tilde{\gamma}\in C^\infty(\Omega)$. For all $(\phi, f) \in \mathcal{A}^{\Omega}_{\partial\omega} \times C_{0}^{\frac{\alpha}{2}; \alpha}([0,T] \times \partial\Omega)$, we  consider the following linear mixed  boundary value problem for a function $u \in C_{0}^{\frac{1+\alpha}{2}; 1+\alpha}([0,T] \times (\overline{\Omega} \setminus \omega[\phi]))$:
{\begin{equation}\label{eq:linbis}
\begin{cases}
    \partial_t u - \Delta u = 0 & \quad\text{in } ]0,T] \times (\Omega \setminus \overline{\omega[\phi]}), 
    \\
    \frac{\partial}{\partial \nu_{\Omega}}  u(t,x) = f(t,x) & \quad \forall (t,x)\in [0,T] \times \partial \Omega, 
    \\
    \frac{\partial}{\partial {\nu_{\omega[\phi]}}}  u(t,x)  +\tilde{\gamma}(x)u(t,x)=0& \quad \forall (t,x)\in [0,T] \times \partial \omega[\phi],
    \\
    u(0,\cdot)=0 & \quad \text{in } \overline{\Omega} \setminus \omega[\phi]\, .
    \end{cases}
\end{equation}}
Clearly, for each $(\phi, f) \in \mathcal{A}^{\Omega}_{\partial\omega}\times C_{0}^{\frac{\alpha}{2}; \alpha}([0,T] \times \partial\Omega)$ problem \eqref{eq:linbis} has a unique solution in $C_{0}^{\frac{1+\alpha}{2}; 1+\alpha}([0,T] \times (\overline{\Omega} \setminus \omega[\phi]))$, which we denote by $u[\phi,f]$. We note that we can rewrite the Robin condition in \eqref{eq:linbis} as 
\[
    \frac{\partial}{\partial {\nu_{\omega[\phi]}}}  u(t,x)  +\tilde{\gamma} \circ \phi (\phi^{-1}(x))u(t,x)=0 \qquad \forall (t,x)\in [0,T] \times \partial \omega[\phi],
\]
for all $\phi \in \mathcal{A}_{\partial \omega}^{\Omega}$. Since the map from $\mathcal{A}_{\partial \omega}^{\Omega}$ to $C^{\frac{\alpha}{2}; \alpha}([0,T] \times \partial\omega)$ which takes $\phi$ to $\tilde{\gamma} \circ \phi$ is of class $C^\infty$ (see Valent \cite[Thm.~4.4]{Va88}), then by exploiting  Theorem \ref{thm:smoothreplin} one can deduce the smoothness of  the map $(\phi,f)\mapsto u[\phi,f]$.
\end{remark}
}

\section*{Acknowledgment}

{The authors would like to thank the anonymous Referee for the valuable comments which have improved the paper.} The authors are members of the ``Gruppo Nazionale per l'Analisi Matematica, la Probabilit\`a e le loro Applicazioni'' (GNAMPA) of the ``Istituto Nazionale di Alta Matematica'' (INdAM).
The authors acknowledge the support  of the
project funded by the EuropeanUnion - NextGenerationEU under the National Recovery and
Resilience Plan (NRRP), Mission 4 Component 2 Investment 1.1 - Call PRIN 2022 No. 104 of
February 2, 2022 of Italian Ministry of University and Research; Project 2022SENJZ3 (subject area: PE - Physical Sciences and Engineering) ``Perturbation problems and asymptotics for elliptic differential equations: variational and potential theoretic methods''. M.D.R., P.L., and P.M. also acknowledge the support of the INdAM GNAMPA Project codice
CUP\_E53C22001930001 ``Operatori differenziali e integrali in geometria spettrale''. P.M. and R.M. also
acknowledge the support from EU through the H2020-MSCA-RISE-2020 project EffectFact, Grant agreement ID: 101008140. 
Part of this work was done while R.M. was visiting C3M - Centre for Computational Continuum Mechanics (Slovenia). R.M. wishes to thank C3M for the kind hospitaliy.

\section*{Conflict of interest}

The authors declare that they  have no conflict of interest.

\section*{Data availability} 

Data sharing is not applicable to this article as no datasets were
generated or analyzed during the current study.

\end{document}